\documentclass[a4paper]{article}

\usepackage[english]{babel}
\usepackage{csquotes}
\usepackage{amsmath}
\usepackage{tikz}
\usepackage{pgf}
\usepackage{pgfplots}
\usepackage{caption}
\usepackage{comment}
\usepackage{subcaption}
\usepackage{amssymb}
\usepackage{mathtools}
\usepackage{svg}
\usepackage{wrapfig}

\usepackage[square,numbers,sort]{natbib}

\usepackage{amsthm}

\newcommand{\derive}[2] {\frac{\partial {#1} }{\partial {#2}}}
\newcommand{\derd}[2] {\frac{\vd #1 }{\vd #2}}

\newcommand{\R}[0]{\mathbb{R}}
\newcommand{\N}{\mathbb{N}}
\newcommand{\of}[1]{\left (#1 \right)}
\newcommand{\abs}[1]{\left | #1\right| }
\newcommand{\mabs}[1] {\,| #1 |}
\newcommand{\norm}[1]{\left \|#1  \right \|}
\newcommand{\CC}{\mathcal{C}}
\newcommand{\set}[2]{ \{ #1 ~|~ #2\}}
\newcommand{\sset}[1]{ \{ #1\}}
\newcommand{\rest}[0]{|}
\newcommand{\Leb}{\mathrm{L}}

\DeclareMathOperator{\ran}{ran}
\newcommand{\vd}{ \mathrm{d}}
\renewcommand{\epsilon}{\varepsilon}
\renewcommand{\phi}{\varphi}
\newcommand{\skp}[1]{\left \langle {#1} \right \rangle}

\DeclareMathOperator{\bigO}{\mathcal{O}}

\theoremstyle{definition}
\newtheorem{definition}{Definition}
\theoremstyle{theorem}
\newtheorem{lemma}{Lemma}
\theoremstyle{remark}
\newtheorem{example}{Example}
\newtheorem{numexp}{Numerical Experiment}
\theoremstyle{theorem}
\newtheorem{theorem}{Theorem}
\newtheorem{remark}{Remark}

\newtheorem{corollary}{corollary}

\usepackage{mathptmx} 

\title{Using the Dafermos Entropy Rate Criterion in Numerical Schemes}
\author{Simon-Christian Klein}

\begin{document}
\maketitle
\begin{abstract}
	The following work concerns the construction of an entropy dissipative finite volume solver based on the convex combination of an entropy conservative and an entropy dissipative flux. We aim to construct a semidiscrete scheme that is entropy stable in the sense of the entropy criterion of Dafermos as well as in the classical sense entropy dissipative. The proposed semidiscrete scheme shows nice properties like $2p$ order accuracy in smooth regions as well as a non-oscillatory behavior around shocks. 
\end{abstract}

\section{Introduction}
\label{intro}
The robustness of numerical methods for hyperbolic conservation laws of the form 
\begin{equation}
\derive{u(x, t)}{t} + \derive{f\circ u (x, t)}{x} = 0 \quad \text{for} \quad u(x, t): \R \times \R \to \R^m \quad \text{with} \quad f: \R^m \to \R^m
\label{eq:hpde}
\end{equation}
is greatly enhanced by numerical methods that do not only approximate \eqref{eq:hpde} but also satisfy entropy inequalities
\begin{equation}
\derive{U \circ u}{t} + \derive{F \circ u}{x} \leq 0.
\label{eq:eie}
\end{equation}
These are used to select one weak solution out of many possible weak solutions. One could further assume that the error for fixed grid size  could be reduced by adhering to entropy inequalities. A scheme has to satisfy \eqref{eq:eie} in a discrete sense
\[
    \frac{U^{n+1}_k - U^{n}_k}{\Delta t} + \frac{F^n_{k+ \frac 1 2} - F^n_{k-\frac 1 2}}{\Delta x} \leq 0
\]
as proposed in \cite{Lax71, Tadmor84I, Tadmor84II} for all or at least one entropy pair $(U, F)$. If solutions of a scheme satisfy all of these inequalities it is called an entropy stable scheme and entropy dissipative if only one entropy inequality is satisfied. Examples of schemes constructed with the aim of being entropy dissipative are for example given in \cite{EntropyViscosity, CvXLimiting, EntropyLimiter, RDI, RDII, DG}. While the objective of this work is also centered around entropy dissipative schemes the motivation stems from an alternative entropy criterion by Dafermos. A second distinction lies in the fact that most of the aforementioned authors construct generalizations of finite element methods while this work is based on classical finite volume methods. We will first look at some numerical artifacts that can still occur with entropy dissipative schemes. Afterwards, a scheme will be constructed that is entropy dissipative and at least approximately satisfies the entropy condition of Dafermos \cite{Dafermos72} and some numerical tests using this scheme will be carried out. Dafermos defined a different entropy criterion using the total entropy in the domain
\[
E_u(t) = \int U \circ u(x, t) \vd x.
\]
A Dafermos entropy solution $u$ is a weak solution that satisfies 
\[
 \forall t > 0: \quad \derd{E_u(t)}{t} \leq \derd{E_{\tilde u}(t)}{t}
\]
compared to all other weak solutions $\tilde u$ of the conservation law \eqref{eq:hpde}. In essence entropy of the solution decreases faster than the entropy of all other solutions.

\section{Comparing schemes by their entropy dissipation}
\label{sec:1}

 Stable high order schemes are often constructed by the addition of suitable dissipation to an at least entropy conservative base scheme, e.g. \cite{AddDiss}. The equation approximated by the resulting scheme is typically of the form
\[
	\derive{u}{t} + \derive{f(u)}{x} = \epsilon \derive{^{2s} u}{x^{2s}}.
\]
The amount of dissipation $\epsilon$ has a strong influence on the resulting errors. To much dissipation results in a simulation of a diffusion (heat) equation while too small amounts of dissipation are responsible for oscillations and can lead to instabilities and order losses. It arises the question of why schemes with low entropy dissipation show bad behavior even if they are formally of high order and entropy dissipative, as they fulfil the entropy inequality for at least one entropy. We would like to shed some light on the connection between the correct amount of entropy dissipation defined by the Dafermos criterion using the following numerical experiments. It should be noted that the Dafermos criterion was designed for  solutions to conservation laws and not their numerical approximations; this means we will look at some numerical solutions and make some assumptions about their limit solutions and their behavior.

\begin{numexp}[Comparing two schemes by the Dafermos criterion]
	A simulation of the Burgers equation \[
		\derive{u}{t} + \frac 1 2 \derive{u^2}{x} = 0
	\] with $u_0(x) = \sin(\pi x)$ was carried out on a periodic domain $ \Omega = [0, 2)$ for $t \in [0, 2]$. The entropy conservative flux of order 4 from \cite{Tadmor87, LMR2002} was used with a dissipation operator due to \cite{Fornberg1998} and $\epsilon = 0.5$ as dissipation coefficient. The numerical solution was compared to a solution calculated by a Godunov scheme. The solution and graphs of the complete entropy in the domain for the quadratic entropy can be seen in figure \ref{fig:godvssbpexp}. Several times a new simulation was started using the solution of the entropy conservative fluxes in conjunction with dissipation as a starting point and the Godunov method as solver. The corresponding total entropy was also plotted in the total entropy diagrams. As we would like to be sure that our conclusions do not depend on the number of points in the domain the simulation was carried out once more with 3000 instead of 100 cells.
	\begin{figure}
		\begin{subfigure}[c]{0.49\textwidth}
			\includegraphics[width=\textwidth]{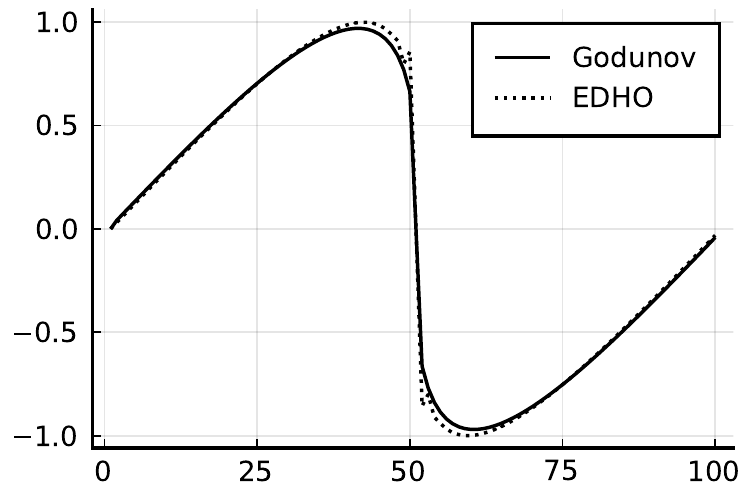}
			\subcaption{Solution at $t = 0.34$ using $N = 100$ points}
		\end{subfigure}
		\hfill
		\begin{subfigure}[c]{0.49\textwidth}
			\includegraphics[width=\textwidth]{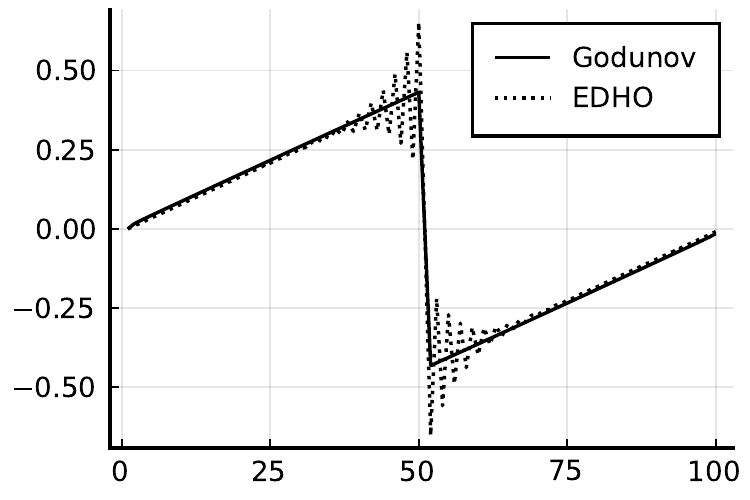}
			\subcaption{Solution at $t = 1.99$ using $N = 100$ points}
		\end{subfigure}
		
		\begin{subfigure}[c]{0.49\textwidth}
			\includegraphics[width=\textwidth]{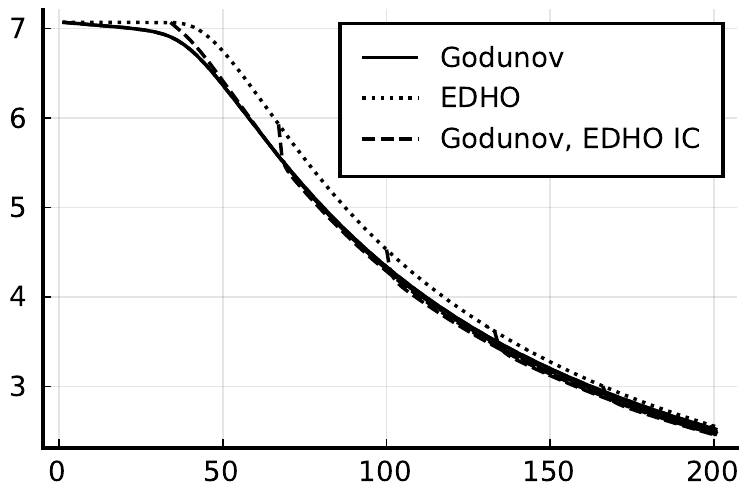}
			\subcaption{Total entropies over time using $N =100$ points}
		\end{subfigure}
		\hfill
		\begin{subfigure}[c]{0.49\textwidth}
			\includegraphics[width=\textwidth]{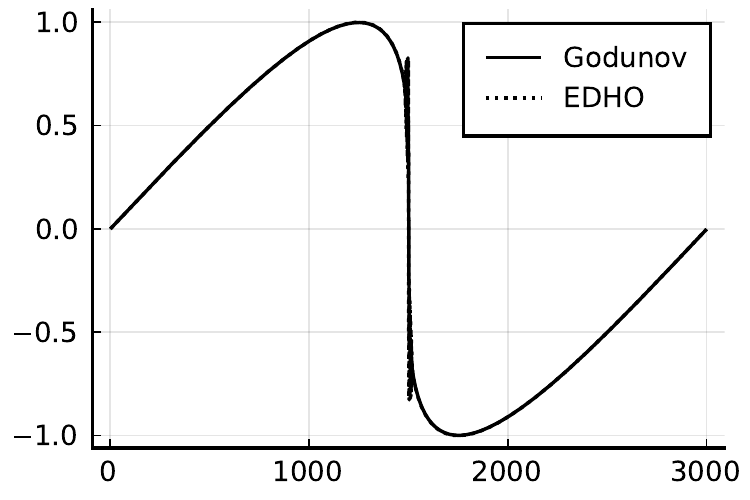}
			\subcaption{Solution at $t =0.33$ using $N=3000$ points}
		\end{subfigure}
		
		\begin{subfigure}[c]{0.49\textwidth}
			\includegraphics[width=\textwidth]{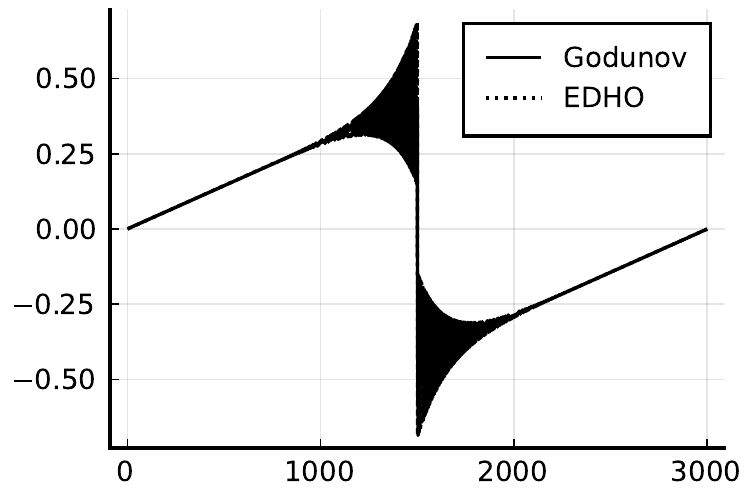}
			\subcaption{Solution at $t = 2.0$ using $N= 3000$ points}
		\end{subfigure}
		\hfill
		\begin{subfigure}[c]{0.49\textwidth}
			\includegraphics[width=\textwidth]{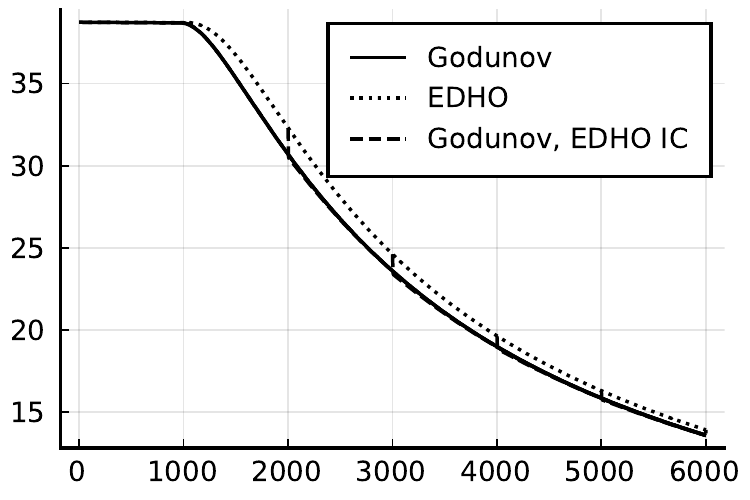}
			\subcaption{Total entropies over time using $N=3000$ points}
		\end{subfigure}
		\caption{Solution to $u_0(x) = \sin(\pi x)$ for the Burgers equation with $N = 100$ cells in the first 3 Graphs and $N=3000$ cells in the last 3 Graphs. A simulation with the Godunov method was started with the solution of the entropy dissipative high order (EDHO) scheme as a starting point at different times.The Godunov method is the basic Godunov scheme with the exact Riemann solver for the Burgers equation and without any reconstruction. Time integration was carried out using a CFL number of $\lambda = 0.5$ and the SSPRK104 scheme.The high order scheme is composed of an entropy stable flux and a dissipation operator. The fourth order entropy conservative flux constructed out of Tadmors entropy conservative flux \cite{Tadmor87} for the Burgers equation and the linear combination developed by LeFloch, Mercier and Rhode \cite{LMR2002} is used. A periodic fourth order dissipation operator with the coefficients given in \cite{Fornberg1998} was used as a dissipation operator with strength $\epsilon = 0.5$. Time integration was, as in the case before, done using the SSPRK104 method.}
		\label{fig:godvssbpexp}
	\end{figure}
	We can clearly see that the entropy dissipative method produces bad results, because oscillations appear around the shock. We can also see that the Godunov scheme dissipates more entropy than the other scheme. The simulations which where carried out by the Godunov method with the solution of the high order method at different times as a starting point are especially interesting. These show a strong reduction of the total entropy until the total entropy of the solution calculated by the Godunov method from the beginning is reached. The Dafermos entropy criterion is only partially applicable in this case as the solutions are approximate solutions. It states in this case that the solution of the high order solver is not the entropy solution, although the solver is technically entropy dissipative, because the negative derivative of the total entropy can even be more negative. It should be noted that the Godunov method on the other hand dissipates entropy even for smooth solutions. This opposes the known theory of hyperbolic conservation laws, as smooth solutions satisfy an entropy equality referred to as an additional conservation law \cite{Dafermos72, Lax71}. The Godunov method satisfies the entropy equality only approximately as the entropy dissipation is small compared to the entropy dissipation after the onset of the shock, but not zero. The Godunov method is still the best possible three point first order method as it is the method with the least possible dissipation that converges to the entropy solution \cite{Tadmor84I, Tadmor84II}.
\end{numexp}

\begin{numexp}[Per cell dissipation of the Godunov method]
	As we saw in the last example the Godunov method leads to a significantly higher total entropy reduction than our high order method, which leads to the question of where this dissipation occurs. 
	\begin{figure}
		\begin{subfigure}{0.49\textwidth}
		\includegraphics[width=\textwidth]{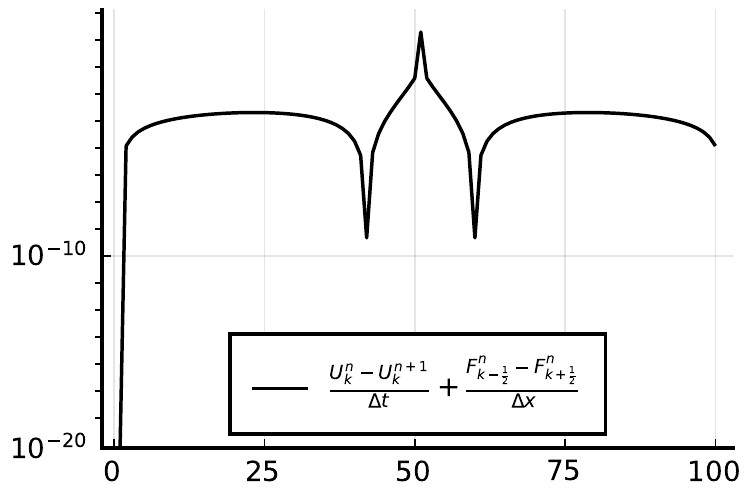}
		\subcaption{Entropy equality violation at $t = 0.34$}
	\end{subfigure}
	\hfill
	\begin{subfigure}{0.49\textwidth}
		\includegraphics[width=\textwidth]{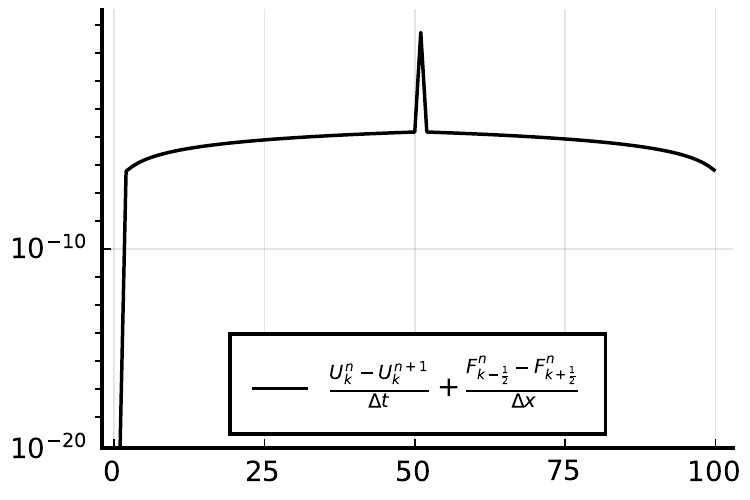}
		\subcaption{Entropy equality violation at $t = 1.99$}
	\end{subfigure}
		\caption{Per cell entropy inequality for the Godunov scheme. The same Godunov method was used as in figure \ref{fig:godvssbpexp} with $N=100$ cells.}
		\label{fig:GodPerCellEntro}
	\end{figure}
	This is why the violation of the entropy equation was plotted for the aforementioned numerical experiment for the Godunov method in figure \ref{fig:GodPerCellEntro}. We can see that a small amount of entropy dissipation occurs during the simulation of a smooth solution while a much bigger amount of entropy dissipation occurs centered around a shock, if present, in compliance with the entropy inequality for shocks. This knowledge was already put to work in \cite{TW2012} using edge sensors. 
	
	\end{numexp}

The last two numerical experiments lead to a new design philosophy for numerical schemes. A good numerical scheme should not only be entropy dissipative in the sense of the entropy inequality. It should also dissipate the correct amount of entropy. This can be governed by the entropy equality for smooth areas, the entropy inequality for shocks and the Dafermos entropy criterion. The Godunov method violates this philosophy by dissipating entropy in smooth areas, while the aforementioned high order method dissipates less entropy than needed and possible around shocks, which violates the Dafermos criterion. It should be noted at the same time that schemes can also dissipate too much entropy in the vicinity of a shock or a maximum. Our proposed scheme will be built out of the following components.
\begin{itemize}
	\item Let the scheme decide if the entropy equality or the entropy inequality holds in an area - this is equivalent to the presence of a shock.
	\item Use an entropy conservative flux if the entropy equality holds.
	\item Dissipate entropy with correct rate in the other case by the use of a dissipative first order flux.

\end{itemize}

Deciding which amount of entropy dissipation is the correct amount is a non-trivial sub-problem. It is not wise to aim for unconstrained maximized entropy dissipation in a numerical method as given by the Dafermos criterion. The reason for this is that the conservation law works as a constraint for the variational formulation of entropy dissipation. Numerical solvers can violate this constraint to some extent and dissipate even more entropy at the cost of higher approximation errors, as even more dissipation leads to larger approximation errors. This is why the highest amount of entropy loss that does not sacrifice low approximation errors is needed.

	It is difficult to find a definition for a suitable amount of entropy loss that does not sacrifice low approximation errors. Godunov's method dissipates the least amount of entropy possible for single conservation laws of all E-fluxes \cite{Tadmor84I, Tadmor84II}, and is thereby a natural candidate. Especially as a high order approximation makes no sense for a region of discontinuity.
\begin{remark}
	One could ask why the Godunov and not an even more dissipative flux like the local Lax-Friedrichs flux should be selected. We are interested in the highest amount of dissipation that does not lower the accuracy, in the sense of the error between numerical approximation and exact (entropy) solution for fixed grid size. While the Lax-Friedrichs method has the same formal order of accuracy the method is less accurate for a fixed grid than the Godunov method and therefore violates our additional constraint. Another perspective can be that a higher entropy dissipation rate than the Godunov method has to be also higher than the entropy rate of the exact solution as the Godunov method uses averages of exact solutions. One could conjecture that such a high dissipation is not possible for any exact weak solution. The Lax-Friedrichs method will be still used in some of the following numerical tests to avoid solving Riemann problems for the Euler equations as the error between exact solution and LF method vanishes for growing grid sizes.\end{remark}

The following chapter is devoted to the construction of the aforementioned solver that tries to satisfy these requirements and uses the Godunov flux as a guide for the correct amount of entropy to dissipate. For simplicity this is done by using an entropy stable first order flux in this case for the dissipation and Tadmor's high order flux in entropy conservative areas of the domain. This scheme thereby should be a numerical scheme that at least approximately satisfies the Dafermos entropy criterion.

\section{The Best of Both Worlds}
\label{sec:2}

We will use entropy conservative fluxes as pioneered in \cite{Tadmor87}. A flux $f$ will be termed entropy conservative if it satisfies a semidiscrete entropy equality
\[
	\derd{U(u_k(t))}{t} = \frac{F(u_{k-p-1}, \dots, u_{k+p}) - F(u_{k-p}, \dots, u_{k+p+1})}{\Delta x}.
\]

\begin{definition}[Convex combination flux]
	We define a new numerical flux by
	\[
		f^{GT}_\alpha \of {u_{i}, u_{i+1}} =  \alpha f^G \of {u_i, u_{i+1}} + (1-\alpha) f^T \of  {u_{i}, u_{i+1}}
	\]
	where $\alpha \in [0, 1] $ is a parameter controlling a convex combination between the Godunov flux presented in \cite{G59, Osher84} and the entropy conservative flux given in \cite{Tadmor87}. The value of \[\alpha = \alpha (u_{i-p+1}, \dots, u_{i+p})\] will in general depend on $u_i$ and therefore the properties of the flux will depend on the selected function $\alpha(u_{i-p+1}, \dots, u_{i+p})$.  
\end{definition}
	It should be clear that this construction does not depend on the use of the Godunov flux. In fact any other numerical flux function could be used, and we will refer to a flux constructed this way using the Lax-Friedrichs scheme as the LFT-Flux and to a flux constructed using the Godunov scheme as the GT flux. Several other entropy conservative fluxes \cite{IsmailRoe2009, ranocha2018thesis} have been constructed for some conservation laws and these can be also substituted for the basic Tadmor entropy conservative flux. 

\begin{lemma}
	The GT-Flux is a consistent and local Lipschitz continuous numerical flux.
	\begin{proof}
		Consistency can be proved by direct insertion.
		\[
			f_\alpha(u, u) =  \alpha f^G \of {u, u} + (1-\alpha) f^T \of  {u, u} = \alpha f(u) + (1-\alpha) f(u) = f(u)
		\]
		We will interpret the arguments of the numerical fluxes as tuples $a = (u_{i}, u_{i+1})$ during the rest of the proof. 
		The Godunov and Tadmor fluxes are Lipschitz continuous with the constants $L_G$ and $L_T$,
		\begin{align*}
			\abs{f_G(a) - f_g(b)} &\leq L_G \norm{a-b} \\
			\abs{f_T(a) - f_T(b)} &\leq L_T \norm{a-b}.
			\end{align*}
		We can conclude using the triangle inequality that the fluxes are also bounded for any bounded subset $U \subset \R^{2p\times m}$
		\begin{align*}
			\forall a \in U,\, \forall I \in \sset{G,T} :\quad \abs{f_I(a)} &\leq \abs{f_I(a) - f_I(a_0)} + \abs{f_I(a_0)} \\ 
			&\leq \abs{f_I(a_0)} + L_I \norm{a-a_0} \\ 
			&\leq \abs{f_I(a_0)} + L_IM_{a_0} = M_I,
		\end{align*}
		where $M_{a_0} > 0$ is any bound that satisfies $\forall a \in U:  \norm{a- a_0} \leq M_{a_0}$ and $a_0 \in U$ is an arbitrary point. Another calculation shows that \[f_\alpha(a) = F(\alpha, a) : [0, 1] \times \R^{2p\times m} \to\R^m\] is a local Lipschitz continuous function
	\begin{align*}
\abs{f_\alpha(a) - f_\beta(b)}	=&\abs{\alpha f_G(a) + (1-\alpha)f_T(a) - \beta f_G(b) - (1-\beta) f_T(b)} \\
= &\abs{\alpha f_G(a) - \beta f_G(b) + (1-\alpha)f_T(a) - (1-\beta) f_T(b)} \\
= &\mabs{ \alpha f_G(a) - \beta f_G(a) + \beta f_G(a) -  \beta f_G(b)  \\
	& +(1-\alpha) f_T(a) - (1-\beta) f_T(a) + (1-\beta)f_T(a) - (1-\beta) f_T(b) } \\
\leq& \abs {\alpha - \beta} \abs{f_G(a)} + \abs{\beta} \abs{f_G(a) - f_G(b)} \\
	&+\abs{\beta - \alpha }\abs{f_T(a)} + \abs{1- \beta}\abs{f_T(a) - f_T(b)}\\
\leq& \abs{\alpha-\beta}(M_G + M_T) + \norm{a-b}(L_G + L_T).
	\end{align*}
		\qed
		\end{proof} 
	\end{lemma}

Using the previous lemma proving that $f_{\alpha\of{u_{i-p+1}, \dots, u_{i+p}}} \of{u_{i}, u_{i+1}}$ is a local Lipschitz continuous flux boils down to proving that $\alpha: \R^{2p\times m} \to [0, 1]$ is local Lipschitz continuous.

We will now prove that this new flux satisfies a semidiscrete entropy inequality at least if there is a cell boundary where $\alpha_{k+\frac 1 2} \neq 0$ holds. The proof is based on cell subdivision, averaging and the convexity of the entropy as already used in \cite{Tadmor84II}.
\begin{theorem}
	The GT flux satisfies the semidiscrete cell entropy inequality
		\[
			\derd{U\circ u_k}{t} \leq \frac{ F^{GT}_{\alpha_{k-\frac 1 2}}(u_{k-1}, u_k) - F^{GT}_{\alpha_{k+\frac 1 2}}(u_{k}, u_{k+1})}{\Delta x}
		\]
		with the numerical entropy Flux
		\[
			F^{GT}_\alpha(u_l, u_r) = \alpha F^G(u_l, u_r) + (1-\alpha) F^T(u_l, u_r)
		\]
		where $F^G(u_l, u_r) = F(u_R(0, u_l, u_r))$ and \[
		\begin{aligned}
		&F^T(u_l, u_r)\\ 
		& =\frac{\langle \derive{U}{u}(u_l) + \derive{U}{u}(u_r),f^T(u_l,u_r) \rangle  + F(u_l) + F(u_r) - \langle \derive{U}{u}(u_l), f(u_l) \rangle  - \langle \derive{U}{u}(u_r), f(u_r) \rangle }{2}
		\end{aligned}
		\] are the respective entropy fluxes of the Godunov\cite{Tadmor84II} and Tadmor fluxes\cite{Tadmor87}.
	\begin{proof}
		We begin our proof by deriving a semidiscrete cell entropy inequality from the discrete cell entropy inequality for the Godunov flux by going over to the limit $\Delta t \to 0$
		\begin{equation*}
		\begin{aligned}
		0 \geq 
		\lim_{\Delta t \to 0} \frac{U(u^{n+1}_{k}) - U(u^n_k)}{\Delta t} - \frac{ F^{G}(u^n_{k-1}, u^n_k) - F^{G}(u^n_k, u^n_{k+1})}{\Delta x}\\ 
		= 
		\derd{U \circ u}{t}  - \frac{ F^{G}(u_{k-1}, u_k) - F^{G}(u_k, u_{k+1})}{\Delta x}\\
		= 
		\skp{ v_k, \derd{u_k}{t} } - \frac{ F^{G}(u_{k-1}, u_k) - F^{G}(u_k, u_{k+1})}{\Delta x} \\
		= \skp{ v_k, \frac{f^{G}(u_{k-1}, u_k) - f^{G}(u_k, u_{k+1})}{\Delta x} } - \frac{F^{G}(u_{k-1}, u_{k}) - F^{G}(u_k, u_{k+1})}{\Delta x}
		\end{aligned}
		\end{equation*}
		The same holds in the sense of an equality also for the product of the entropy variable with the Tadmor flux. We first look at the special case $\alpha_{k-\frac 1 2} = \alpha= \alpha_{k + \frac 1 2}$ and use the entropy variable $v_k = \derive {U \circ u}{u}\rest_{u_k}$ to find
		\[
			\begin{aligned}
			&\derd{U\circ u_k}{t}\\ =& \skp{v_k, \derd{u}{t}}\\
			 =& \skp{v_k, \frac{\alpha f^G(u_{k-1}, u_{k}) + (1-\alpha) f^T(u_{k-1}, u_k) - \alpha f^T(u_{k}, u_{k+1}) - (1-\alpha)f^G(u_k, u_{k+1})}{\Delta x}} \\
			=& \alpha\skp{ v_k, \frac{f^G(u_{k-1}, u_k) - f^G(u_k, u_{k+1})}{\Delta x}} + (1-\alpha) \skp{v_k, \frac{f^T(u_{k-1}, u_k) - f^T(u_k, u_{k+1})}{\Delta x}} \\
			\leq& \alpha \frac{F^G(u_{k-1}, u_k) - F^G(u_k, u_{k+1})}{\Delta x} + (1-\alpha) \frac{F^T(u_{k-1}, u_k) - F^T(u_k, u_{k+1})}{\Delta x}\\
			 =& \frac{F^{GT}_{\alpha}(u_{k-1}, u_k) - F^{GT}_{\alpha}(u_{k}, u_{k+1})}{\Delta x}.
			\end{aligned}
		\]
		Furthermore, we now consider the general case $\alpha_{k-\frac 1 2} \neq \alpha_{k +\frac 1 2}$ under usage of the first case.
		 The derivative of the average $u_k$ can be rewritten as the average of two schemes for the averages $u_{k-\frac 1 4}$ and $u_{k+\frac 1 4}$ 
		 \[\begin{aligned}
		 	\derd{u_{k - \frac 1 4}}{t} &= \frac{f_{\alpha_{k - \frac 1 2}}(u_{k-1}, u_k) - f_{\alpha_{k - \frac 1 2}}(u_{k}, u_k)}{\Delta x/2} \\ \derd{u_{k + \frac 1 4}}{t}&= \frac{f_{\alpha_{k + \frac 1 2}}(u_{k}, u_k) - f_{\alpha_{k + \frac 1 2}}(u_k, u_{k+1})}{\Delta x/2},
		 	\end{aligned}
		 \]
		 that can be thought of as the cell subdivision in figure \ref{fig:celldivision}
		\[
			\begin{aligned}
			\derd{u_k}{t} = &\frac{f_{\alpha_{k - \frac 1 2}}(u_{k-1}, u_k) - f_{\alpha_{k + \frac 1 2}}(u_k, u_{k+1})}{\Delta x} \\
			=&\frac 1 2 \left ( \frac{f_{\alpha_{k - \frac 1 2}}(u_{k-1}, u_k) - f_{\alpha_{k - \frac 1 2}}(u_{k}, u_k)}{\Delta x/2} 
			+ \frac{f_{\alpha_{k + \frac 1 2}}(u_{k}, u_k) - f_{\alpha_{k + \frac 1 2}}(u_k, u_{k+1})}{\Delta x/2} \right ) \\
			=& \frac{\derd{u_{k-\frac 1 4}}{t} + \derd{u_{k+\frac 1 4}}{t}}{2}.
			\end{aligned}
		\]
		This is the semidiscrete equivalent of the cell division usually employed to make use of the convexity of the entropy. In our case this allows us to change from $\alpha_{k - \frac 1 2}$ to $\alpha_{k + \frac 1 2}$ as our fluxes are consistent, namely $f^{GT}_{\alpha_{k - \frac 1 2}}(u,u) = f(u) = f^{GT}_{\alpha_{k - \frac 1 2}}(u,u)$. This implies together with the consistency of the entropy fluxes
		\begin{equation}
			\begin{aligned}
			\derd{U \circ u_k}{t} = &\skp{v_k, \derd{u_k}{t}} = \frac 1 2 \left (\skp{v_k, \derd{u_{k-\frac 1 4}}{t} } +  \skp{v_k, \derd{u_{k+\frac 1 4}}{t}} \right)\\
			 \leq&
			  		\frac 1 2\left(\frac{F^{GT}_{\alpha - \frac 1 2}(u_{k-1}, u_k) - F^{GT}_{\alpha -\frac 1 2}(u_{k}, u_{k})}{\Delta x/2} 
			  		+ \frac{F^{GT}_{\alpha + \frac 1 2}(u_{k}, u_k) - F^{GT}_{\alpha + \frac 1 2}(u_{k}, u_{k+1})}{\Delta x/2} \right)
		  		\\
		  	= & 
		  	\frac{F^{GT}_{\alpha_{k-\frac 1 2}}(u_{k-1}, u_k) - F^{GT}_{\alpha_{k+\frac 1 2}}(u_{k}, u_{k+1})}{\Delta x}
			  \end{aligned}
			  \label{eq:GTeiq}
		\end{equation}
		and completes the Proof. \qed
	\end{proof}

			\begin{figure}
				\centering
				\begin{tikzpicture}
					\draw (-5, 0) -- (5, 0); 
					\draw (-5, 0) node [left] {\large$t^n$};
					\draw (-1, 0) -- (-1, 2); 
					\draw (-1, 2) node [above]{ \large$f_{k - \frac 1 2}$};
					\draw (-1, 0) node [below] {\large$x_{k - \frac 1 2}$};
					\draw (1, 0) -- (1, 2); 
					\draw (1, 2) node [above]{\large $ f_{k + \frac 1 2}$};
					\draw (1, 0) node [below] {\large$x_{k + \frac 1 2}$};
					\draw [dotted](0, 0) -- (0, 2); 
					\draw (0, 2) node [above] {\large$f_k$};
					\draw (0, 0) node[below] {\large$x_k$};
					\draw (-3, 0) -- (-3, 2); 
					\draw (-2, 0) node [below] {\large$x_{k-1}$};
					\draw (3, 0) -- (3, 2);
					\draw (2, 0) node [below] {\large$x_{k+1}$};
					\draw (-2, 1) node {\large$u_{k-1}$};
					\draw (2, 1) node {\large$u_{k+1}$};
					\draw (0, 4/3) node {\large$u_k^n$};
					\draw (0, 2/3) node {\large$u_{k-\frac 1 4}\, u_{k+ \frac 1 4}$};
				\end{tikzpicture}
				\caption{The subdivision of a cell in space, initialized with the mean value of the old cell.}
				\label{fig:celldivision}
			\end{figure}
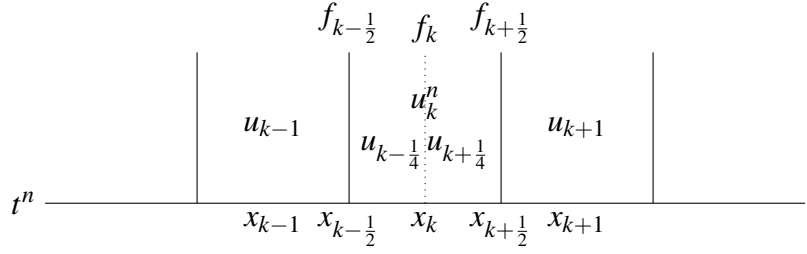

\end{theorem}
The aforementioned arguments show that our flux is entropy dissipative in the usual sense if $\alpha$ is chosen to be nonzero.

In \cite{LMR2002}[Section 4.1] the entropy conservative flux of Tadmor was extended via the usage of linear combinations into an entropy conservative flux of order $2p$. We will also use this idea on our flux. As we have already used $\alpha$ as a parameter for convex combinations we will use $c^r_p$ instead to denote the coefficients in the linear combination.
\begin{definition}
	 We define the high order LMRGT flux of order $2p$ as 
		\[
			f^{LMRGT}_{\alpha}(u_{k-p+1}, \dots , u_{k+p}) = \sum_{r=1}^p c^r_p (f^{GT}_{\alpha}(u_k, u_{k+r}) +  \dots + f^{GT}_{\alpha}(u_{k-r+1}, u_{k+1})).
		\]
\end{definition}
	It follows from the definition that this flux is of order $2p$ for $\alpha_k = 0$ i.e. when the entropy equality holds. One further deduces from
	\[
		\begin{aligned}
		&f^{LMRGT}_{\alpha}(u_{k-p+1}, \dots , u_{k+p}) - f^{LMRGT}_{0}(u_{k-p+1}, \dots , u_{k+p}) \\
		=&\alpha \sum_{r=1}^p c^r_p ((f^{G}(u_k, u_{k+r})-f^T(u_k, u_{k+r})+ \dots + f^{G}(u_{k-r+1}, u_{k+1})-f^{T}(u_{k-r+1}, u_{k+1})) \\
		=&\alpha \bigO(\Delta x),
		\end{aligned}
	\]
	that the scheme is even of order $2p$ if the linear combination for an $2p$ order accurate entropy conservative flux is used in the construction and $\alpha(u_{k-q+1}, \dots, u_{k+q}) = \bigO((\Delta x)^{2p-1})$ holds.

 While one aims for a discrete entropy inequality we are at least able to proof a semidiscrete cell entropy inequality for this flux
 \begin{corollary}
	The semidiscrete scheme
		\[
		\begin{aligned}
			\derd{u_k}{t} =& \frac{f^{LMRGT}_{\alpha_{k- \frac 1 2}}(u_{k-p}, \dots, u_{k+p-1}) - f^{LMRGT}_{\alpha_{k + \frac 1 2}}(u_{k - p +1}, u_{k+p}) }{\Delta x}\\
			 =&- \sum_{r=1}^p c_p^r \frac{f^{GT}_{\alpha_{k + \frac 1 2}}(u_{k}, u_{k+r}) - f^{GT}_{\alpha_{k - \frac 1 2}}(u_{k-r}, u_k)}{\Delta x}
			 \end{aligned}
		\]
		satisfies a semidiscrete entropy inequality 
		\[
			\derd{U\circ u_k}{t} \leq \frac{F^{LMRGT}_{\alpha_{k- \frac 1 2}}(u_{k-p},\dots, u_{k+p-1}) - F^{LMRGT}_{\alpha_{k + \frac 1 2}}(u_{k-p+1},\dots, u_{k+p}) }{\Delta x}
		\]
		with an consistent numerical entropy flux given by
		\[
			F^{LMRGT}_{\alpha}(u_{k-p+1}, \dots, u_{k+p}) = \sum_{r=1}^p c_p^r(F^{GT}_\alpha(u_{k}, u_{k+r}) + \dots + F^{GT}_\alpha(u_{k-r + 1}, u_{k+1}))
		\]
		if $\forall k: \alpha_{k + \frac 1 2} \in (0,1]$ holds.
		\begin{proof}
			We follow the proof of the semidiscrete entropy inequality from \cite{LMR2002}[Section 4.1] and multiply the definition of the scheme by the entropy variable $v_k$ to find
			\begin{align*}
				\derd{U\circ u_k}{t} 
				= 
				\langle v_k, \derd{u_k}{t}\rangle  
				=&  
				 \skp{v_k,  \sum_{r=1}^p c^r_p \frac{f^{GT}_{\alpha_{k-\frac 1 2}}(u_{k-r}, u_{k}) - f^{GT}_{\alpha_{k + \frac 1 2}}(u_{k}, u_{k+r})}{\Delta x} } \\
				= &
				\sum_{r=1}^p c^r_p \skp{ v_k, \frac{f^{GT}_{\alpha_{k - \frac 1 2}}(u_{k-r}, u_{k}) - f^{GT}_{\alpha_{k + \frac 1 2}}(u_{k}, u_{k+r})}{\Delta x} } \\
				 \overset{\eqref{eq:GTeiq}}{\leq}&
				\sum_{r=1}^p c^r_p  \frac{F^{GT}_{\alpha_{k-\frac 1 2}}(u_{k-r}, u_{k}) - F^{GT}_{\alpha_{k + \frac 1 2}}(u_{k}, u_{k+r})}{\Delta x}\\
				= &
				\frac{F^{GTLMR}_{\alpha_{k - \frac 1 2}}(u_{k-p}, u_{k+p-1}) - F^{GTLMR}_{\alpha_{k + \frac 1 2}}(u_{k-p+1}, u_{k + p})}{\Delta x}.
			\end{align*}
		\end{proof} \qed
\end{corollary}

 Our first numerical experiment showed that an entropy dissipative scheme alone is not enough to guarantee good approximate solutions. This is why we will now construct an algorithm to find values for $\alpha$ to control our flux according to the Dafermos entropy criterion.

\begin{definition}
	We call $\alpha: \R^{2p\times m} \to [0, 1]$ an entropy inequality predictor with a $(2p)$ point stencil if
	\begin{align*}
&\lim_{h \to 0}	\alpha(u_{i-p+1}, \dots, u_{i +p}) \\ = &\begin{cases} 0 &  \exists x \in [x_i -(p-1)\Delta x, x_i + p\Delta x]: \derive{U\circ u}{t} + \derive{F \circ u}{x} < 0\\
	1 & \forall x \in [x_i -(p-1)\Delta x, x_i + p\Delta x]:\derive{U\circ u}{t} + \derive{F \circ u}{x} = 0 \end{cases}
	\end{align*}
	holds for the complete stencil. The input values $u_{i-p+1}, \dots, u_{i+p}$ shall be the mean values of the solution in the respective cells as present in a Finite Volume solver. We will call the entropy inequality predictor slope limited if 
	\[ \abs{\alpha_i - \alpha_{i+1}} < M \quad \text{with} \quad \alpha_i = \alpha(u_{i-p+1}, \dots, u_{i +p})\]
	holds for some $M < 1$ and all $i$.
\end{definition}

	The slope limiting property was inspired by the idea to limit the slope of $\alpha$ with respect to the grid index $i$. This should not be mixed up with a bound on the slope of $\alpha$ with respect to $x$. Such a bound would be scaling with the distance between $x_i$ and $x_{i+1}$ as present in the usual definition of a difference quotient. This ensures that $\alpha$ switches between $0$ and $1$ over several mesh points while the size of this switch is scaled down with respect to the physical scale for a finer grid. The switch needs at least $\lfloor1/M \rfloor$ points.

\begin{lemma}[Smoothstep \cite{Perlin2002}]\label{lem:sm}
	The function 
		\begin{equation*}
			H_{sm}(x) = \begin{cases}
				0 & x  \leq 0 \\
				6x^5 - 15x^4+10x^3 & 0 \leq x  \leq 1\\
				1 & 1 \leq x \\
			\end{cases}
		\end{equation*}
	is a $\CC^2$ function with zero first and second derivatives at $x = 0$ and $x = 1$.
\end{lemma}
\begin{figure}
	\centering
	\begin{tikzpicture}\begin{axis}[]
	\addplot[black,domain=0:1,samples=201,]
	{6*x^5-15*x^4+10*x^3};
	\end{axis}
	\end{tikzpicture}
	\caption{Plot of the smooth step function.}
	\end{figure}
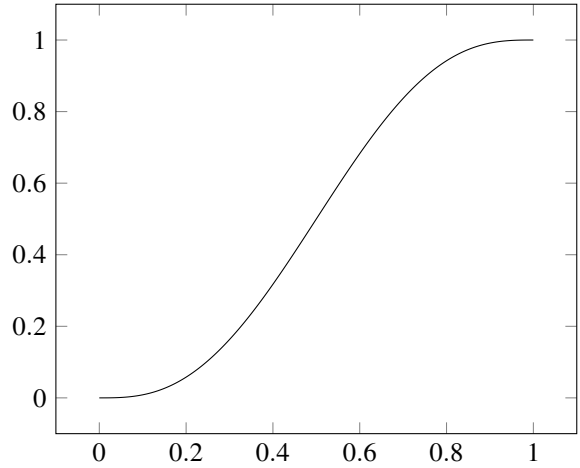

We will need some special operations on functions for the construction of our predictor which are motivated by mollification. The convolution \cite[p. 216]{LaxFun} of $f$ and $g$ is defined as
\[
f * g(x) = \int_\Omega f(y) g(x-y) \vd y
\]
for suitable $f$ and $g$. Please note that this can be interpreted as the integral of all combinations of $f(\cdot)$ with $g(- \cdot)$ multiplied over $\R$ and indexed by their respective shift between the argument of $f$ and $g$. A mollification, first defined in \cite{Friedrichs1944Molli}, is a convolution of a function $f$ with a suitable $g$ giving a smoother function as $f$.  The result \cite[p. 216]{LaxFun}
\begin{align*}
\norm{f*g}_1 \leq \norm{f}_1 \norm{g}_1,
\end{align*}
relates the norm of the mollified function $f * g$ to the original function.
 As convolution is coupled to (Lebesgue)-integration which in turn provides the Lebesgue norms, one can ask if also an equivalent of convolution for other norms exists. We will find such an equivalent with interesting properties for our application related to the uniform norm $\norm \cdot_\infty$.

\begin{definition}[Minkowsky sum and Minkowsky product]
	Given two sets $A \subset \R$ and $B\subset \R$ we define the Minkowsky sum and Minkoswky product as
	\[A \oplus B = \set{a + b}{a\in A, b \in B}\]
	and
	\[A \odot B = \set{ab}{a \in A, b \in B}.\]
	\end{definition}

\begin{lemma}[Special properties of the Minkowsky product and sum]
	\label{lem:supaddhom}
	Let $A, B \subset \R$ be two sets. In this case
\[ \sup (A \oplus B) = \sup A + \sup B\]
	holds. If additionally $\forall a \in A: a \geq 0$ and $\forall b \in B: b\geq 0$ hold the equality
\[	\sup (A \odot B) = \sup A \cdot \sup B
\]
is also satisfied. In other words the supremum is additive and positively homogeneous for sets. 

\end{lemma}

\begin{definition}[Minkowsky product of functions]
	Let $f, g: \R \to \R$. Their Minkowsky product is a map $f \odot g: \R \to \R$, defined by 
	\[
	(f \odot g)(x) = \set{f(x-y)g(y)}{y \in \R}.
	\]
	In other words the Minkowsky product of $f$ and $g$ at $x$ is the set of all values of $f$ multiplied by $g$ so that the argument added together gives $x$. Compare this to the integrand of the convolution.
	\end{definition}

\begin{definition}
	Let $f, g: \R \to \R$ be bounded real maps. The sup-mollification of $f$ and $g$ is defined as
	\[
	f \circledast g (x)= \sup f \odot g(x) = \sup_{y \in \R}(f(x-y)g(x)).
	\]
	Please note that $f \odot g$ is a set depending on $x$ and the supremum is not taken over $x$, but over the set at the point $x$.
	\end{definition}

We will use the defined sup-mollification operator to ensure the slope limiting property of our entropy inequality predictor. We will now prove some useful lemmas that will also show that our entropy inequality predictor is Lipschitz continuous and keeps $\alpha$ up at one in a region around an entropy dissipating shock. 

\begin{lemma}
	\label{lem:mollnormeq}
	Let $f, g: \R \to \R_{\geq 0}$ be bounded functions. In this case
	\[ \sup_x f \circledast g(x) = \sup_z f(z) \sup_z g(z) \]
	holds.
	\end{lemma}
\begin{proof}
	An easy calculation shows
	\begin{align*}
	\sup_x f\circledast g &= \sup_x \, \sup_y f(y)g(x-y) = \sup_{(x, y)\in \R^2} f(y) g(x-y) \\
	&= \sup_{(x, y)\in \R^2}f(y)g(x) = \sup \ran f \odot \ran g \\
	&= (\sup \ran  f) \cdot  (\sup \ran g).
	\end{align*}\qed
	\end{proof}

\begin{lemma}
	\label{lem:supdiff}
	Let $f, g: \R \to \R$ be bounded functions. Then 
	\[
		\abs{\sup_x f(x) - \sup_y g(y)} \leq \sup_x \abs{f(x)-g(x)}
	\]
	holds.
	
	\begin{proof}
		We start by stating that 
	\[
		\sup_x f(x) = \sup_x g(x) + f(x) - g(x) \leq \sup_y g(y) + \sup_x f(x)-g(x)
	\]
	holds. This can be rearranged and bounded so that
	\[
		\sup_x f(x) - \sup_y g(y) \leq \sup_x f(x)-g(x) \leq  \sup_x \abs{f(x)-g(x)}
	\]
	holds. As this also holds if the roles of $f, g$ are swapped, it follows
	\[
		\abs{\sup_x f(x)- \sup_y g(y)} \leq \sup_x \abs{f(x)-g(x)}.
	\]
	\qed
	\end{proof}
	\end{lemma}

\begin{lemma}[Sup-mollification is a Lipschitz continuous operator]
	For bounded $f_1, f_2, g: \R \to \R$ it holds
	\[
		\norm{f_1 \circledast g - f_2 \circledast g}_\infty \leq \norm{g}_\infty\norm{f_1 - f_2}_\infty  
	\]
	\begin{proof}
		We use lemma \ref{lem:supdiff} and \ref{lem:mollnormeq} to prove
		\begin{align*}
				\norm{f_1 \circledast g - f_2 \circledast g}_\infty &= \sup_x \, \abs {\sup_y \, f_1(y)g(x-y) - \sup_y \, f_2(y)g(x-y)} \\
				&\overset{\text{lem \ref{lem:supdiff}}}\leq \sup_x \, \sup_y \, \abs{f_1(y)g(x-y) - f_2(y)g(x-y)} \\
				&= \sup_x \, \sup_y \abs{f_1(y) - f_2(y)} \abs{g(x-y)}\\
				&\overset{\text{lem \ref{lem:mollnormeq}}}{=} \sup_x \, \abs{f_1(x)-f_2(x)}\sup_y\abs{g(y)}\\
				& = \norm{g}_\infty \norm{f_1 - f_2}_\infty.
		\end{align*}
	\qed
	\end{proof}
	\end{lemma}

\begin{lemma}[Slope condition inequality]
	Let $f, g: \R \to \R$ be bounded functions. If $g$ satisfies for a fixed $h \in \R$
	\[
	 \exists M \in \R : \sup_{x \in \R} \, \abs {g(x + h) - g(x)} \leq M,
	\]
	then the sup-mollification $f \circledast g$ full fills
	\[
		\abs{f\circledast g(x+h) - f\circledast g(x)} \leq M \cdot \sup \, \abs{f(y)}.
	\]
	\begin{proof}
		We again use lemma \ref{lem:supdiff} to prove 
		\begin{align*}
				\abs{f\circledast g(x+h) - f\circledast g(x)} & = \abs{\sup_y f(y)g(x+h-y) - \sup f(y)g(x-y)} \\
				&\leq \sup_y \abs{f(y)g(x+h-y) - f(y)g(x-y)} \\
				& = \sup_y \abs{f(y)}\abs{g(x+h-y) - g(x-y)} \\
				&\leq \sup_y \abs{f(y)} \sup_y \abs{g(x+h-y) - g(x-y)}\\
				& \leq M \cdot \sup \abs{f}.
		\end{align*}
			\qed
	\end{proof}
	\end{lemma}

\begin{lemma}[Plateau condition inequality]
	Let $f, g: \R \to \R$ be bounded,  $x_0 \in \R$ and $ \epsilon > 0$. Then
\[
	\forall x \in [-\epsilon, \epsilon]: g(x) > c \in \R
\] implies 
\[
	\forall x \in [x_0 - \epsilon, x_0 + \epsilon]: f \circledast g(x) \geq c f(x_0).
\] 
\begin{proof}
	Let $x \in [x_0 - \epsilon, x_0 + \epsilon]$. If we set $y = x_0$ it follows $x-y \in [-\epsilon, \epsilon]$ and
	\begin{align*}
		f(y)g(x-y) = f(x_0) g(x-y) \geq f(x_0) c
		\\ \implies f(x_0) c \leq \sup_y f(y)g(x-y) = f \circledast g(x).
	\end{align*}
	\qed
\end{proof}
	\end{lemma}

\begin{definition}[Discrete sup-mollification]
	Let for $n \in \N$ be the vector space of step functions on $[0, 1]$ denoted as
	\[
	S_n = \set{f:[0, 1] \to \R}{\forall i = 0, \dots, n-1: f\rest_{[i/n, (i+1)/n]} = f_i \in \R}.
	\]
	We can further define an embedding $Z:S_n \to S$ of this space into the step functions over $\R$, denoted as $S$, by
	\[
		Z: S_n \to S, f \mapsto Zf, Zf(x) = \begin{cases}f(x)& x \in [0, 1] \\ 0 & \text{else} \end{cases}.
	\]
	These two definitions allow us to define the discrete sup-mollification of $f, g \in S_n$ as
		\[
		(f \circledast g)\rest_{[i/n, (i+1) / n]} = (Zf \circledast Zg)\rest_{[i/n, (i+1) / n]} = \max_{j \in \sset{0, \dots, n-1}} \tilde f_j \tilde g_{i-j}. \quad \text{for } i = 0, \dots, n-1.
		\]
	The values $\tilde f_i \in \R$ and $\tilde g_i \in \R$ relate to $f_i$ and $g_i$ as 
	\[
		\tilde f_i = \begin{cases}f_i & i \in \{0, \dots, n-1\}  \\ 0 & \text{else}\end{cases} \quad \tilde g_i = \begin{cases}g_i & i \in \{0, \dots, n-1\}  \\ 0 & \text{else} \end{cases}.
	\]
	The sup mollfication for step functions on the interval $[a, b]$ shall be defined using the coordinate transform $\phi(x) = \frac{x-a}{b-a}$ and the corresponding inverse $\phi^{-1}(y) = a + y(b-a)$. Using this transform yields
	\[
		(f \circledast g)(x) = (Z(f \circ \phi) \circledast Z(g \circ \phi)) \circ \phi^{-1}(x).
	\]
	\end{definition}

	It is also possible to exactly sup-mollify piecewise linear functions.

\begin{example}[The Godunov Flux entropy inequality predictor]
	A suitable entropy inequality predictor can be constructed from the Godunov flux by looking at it's entropy dissipation
	\begin{equation*}
		s^n_k = \frac{F(u^n_{k+1}, u^n_{k}) - F(u^n_k, u^n_{k-1})} { \Delta x} + \frac{U(u_k^{n+1}) - U(u^{n}_{k})} { \Delta t} \leq 0
	\end{equation*} as given in \cite{Tadmor84I, Tadmor84II}.
	As the Godunov scheme is entropy stable $s^n_k \leq 0$ holds $\forall n, k$. We can use this value to predict if the entropy equality holds - or the inequality. A problem occurs as in fact $s^n_k <0$ is even true for smooth initial conditions like $u_0(x) = \sin(\pi x)$ in the first time step. 
	
	Similar Problems appear in the context of edge sensors and local viscosity \cite{AGY2004, TW2012} and are usually solved by a thresholding process. In our case this threshold will be carried out by the smooth step function $H_{sm}$ from lemma \ref{lem:sm} to ensure a Lipschitz-continuous transition. As there are in fact two free parameters in this approach, one for determining the lower threshold, and another to control the width of the smooth step, it is imperative to find parameters that are at least independent of the used grid. We therefore define
	\[
		u^+ = \max_k u_k^n \quad u^- = \min_k u_k^n
	\]
	for single conservation laws and
	\[
		u^+ = u(\arg \max_x U \circ u(x, t), t) \quad u^- = u(\arg \min_x U\circ u(x, t), t)
	\]
	as the cell values having maximum and minimum entropy in the domain for systems of conservation laws. These values can be afterwards used to construct the Riemann problem with the initial conditions
	\[
		u_1(x, 0) = \begin{cases}u^- \\ u^+ \end{cases} \quad u_2(x, 0) = \begin{cases}u^+ & x< 0\\u^- & x \geq 0\end{cases}.
	\]
	By looking at their entropy dissipation $s^n_k$ when the Godunov scheme is applied one finds a reference 
	\[
		s^\text{ref} = \min\left(\min_k s^n_k(u_1(\cdot, 0)), \min_k s^n_k(u_2(\cdot, 0)) \right)
	\] 
	for the entropy dissipation of a strong shock that could be present in the solution. While this approach involves a lot of hand waving the numerical results are quite satisfactory and further research could be centered around this issue. The values
	\[
		r^n_k = H_{sm}\of {\frac{\frac {s^n_k} {s^{\text{ref}}} - a }{b}}
	\]
	 depend smoothly on $s^n_k$ but can still have extremely localized spikes as wide as only a few cells. The parameter $a \in \R$ is a threshold under which the result of the entropy inequality predictor should be thought of as zero, while $b \in \R$ corresponds to a typical amplitude of a spike in the entropy dissipation indicating a shock. Numerical tests indicate that a instantaneous switching between fluxes leads to undesired oscillations around their interface. Furthermore, the stencil of the high order Tadmor flux is wider than the stencil of the Godunov scheme and the derivation of the high order Tadmor scheme assumes an entropy conservative solution in its derivation, which will be violated by an entropy dissipating discontinuity in the solution. These two problems are considered in the definition of the entropy inequality predictor. Responsible are the slope limiting property and the definition as a scale for the violation of the entropy equality on the entire stencil of a scheme. In this case the wider stencil of the high order modified Tadmor scheme is relevant. We will satisfy these requirements using sup-mollification of $r^n_k$, i.e. its associated piecewise constant function, and a suitable kernel. We chose the cut hat function
	\[
	h(x) = \max(0, \min(1, 2x+2, -2x + 2))
	\]
	in a properly rescaled fashion for this purpose. 
		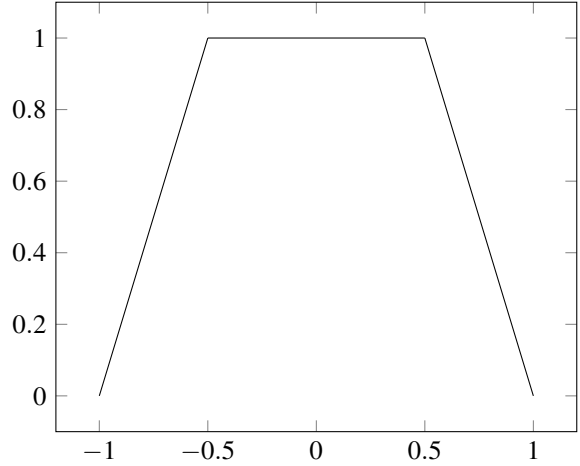
\begin{figure}
			\centering
		\begin{tikzpicture}\begin{axis}[]
		\addplot[black,domain=-1:-0.5,samples=201,]
		{2*x+2};
		\addplot[black, domain=-0.5:0.5, samples = 201]
		{1};
		\addplot[black, domain=0.5:1, samples=201,]
		{-2*x+2};
		\end{axis}
		\end{tikzpicture}
		\caption{Plot of a cut hat function $h(x)$.}
	\end{figure}
	 and define
	\[
	\alpha = r \circledast h.
	\]
	Other choices are possible, and it is not clear yet if a smoother mollifier improves the scheme.

\end{example}

\begin{lemma}
	The Godunov Flux inequality predictor 
	\[
	\alpha^n = H_{sm}\of {\frac{\frac {s^n_k} {s^{\text{ref}}} - a }{b}} \circledast h
	\]
	is slope limited. 
	\begin{proof}
			This follows from the fact that 
			\[\ran H_{sm}\of {\frac{\frac {s^n_k} {s^{\text{ref}}} - a }{b}} \subset [0, 1]\]
			holds and the cut hat function has limited slope using the slope condition inequality.\qed
		\end{proof}
		
	\end{lemma}
While the aforementioned entropy inequality predictor is able to deliver satisfactory results the solution of a Riemann problem, needed to calculate the Godunov flux, is not always easily obtained. This is why two other entropy inequality predictors, one of theoretical and one also of practical value, were constructed. 
\begin{example}[The Lax-Friedrichs entropy inequality predictor]
	The aforementioned construction can be also applied to the Lax-Friedrichs scheme, and it's corresponding entropy inequality and entropy flux, proved in \cite{Lax71, Tadmor84I, Tadmor84II}.
	The entropy flux is given by
	\begin{equation*}
		F(u_l, u_r) = \frac{F(u_l) + F(u_r)}{2} + \frac{U(u_l) - U(u_r)}{2 \lambda},
	\end{equation*}
leading to the entropy production
	\begin{equation*}
	s^n_k = \frac{F(u^n_{k+1}, u^n_{k}) - F(u^n_k, u^n_{k-1})} { \Delta x} + \frac{U(u_k^{n+1}) - U(u^{n}_{k})} { \Delta t} \leq 0.
	\end{equation*}
	Entering these results leads to the entropy inequality predictor
	\[
	\alpha^n = H_{sm}\of {\frac{\frac {s^n_k} {s^{\text{ref}}} - a }{b}} \circledast h.
	\]
\end{example}
Sadly, while this predictor has a rigorous provable background from \cite{Lax71}, its practical use is complicated. The line between an area of entropy conservation and an entropy dissipating shock is blurred by the big amount of dissipation present in the Lax-Friedrichs scheme that also happens in smooth areas. This is why a second entropy inequality predictor was constructed with reduced dissipation. This reduction is based on linear reconstruction in an ENO type fashion \cite{ENOIII}.

\begin{example}[The ENO2 Lax Friedrichs entropy inequality predictor]
	Given an piecewise constant solution $u_k^n$ one first calculates the piecewise linear reconstructions
	\[
		\tilde u_k^n(x) = \begin{cases}
			u_k^n + a_l(x-x_k) & a_l < a_r \\
			u_k^n + a_r(x-x_k) & a_r \geq a_l \\
		\end{cases} \quad a_l = \frac{u_k^n - u_{k-1}^n}{x_{k} - x_{k-1}} \quad a_r = \frac{u_{k+1}^n - u_{k}^n}{x_{k+1} - x_k}.
	\]
	Using this reconstruction directly in a finite volume entropy inequality is not possible, as the scheme
	\[
		u_k^{n+1} = u_k^n + \lambda\left(f\left(\tilde u^n_{k-1}\left(x_{k- \frac 1 2}\right),\tilde u^n_k\left(x_{k - \frac 1 2}\right)\right) - f\left(\tilde u^n_k\left(x_{k+ \frac 1 2}\right),\tilde u^n_{k+1}\left(x_{k + \frac 1 2}\right)\right) \right)
	\]
	has to the authors knowledge no known entropy fluxes. We instead seek to calculate an approximation of the entropy dissipation of this reconstructed solution by using it as the initial condition for a first order Lax-Friedrichs solver. It is sufficient to use this solver at points of discontinuity as the entropy equality holds for the smooth areas of the solution. We therefore use the subdivision of our primary cells sketched in figure \ref{fig:LFcelldiv} to start the Lax-Friedrichs method. Let $x_{k+1/2}$ be the cell boundary between the cell around $x_k$ and $x_{k+1}$ and $\frac{\Delta x}{6} \geq \epsilon > 0$ an arbitrary parameter for a sub-cell size. We introduce new cell boundaries at 
	\[x_l^-=x_{k+\frac 1 2} - 3\epsilon \quad 
		x_l^+=x_m^- = x_{k+\frac 1 2} - \epsilon \quad
		x_m^+ = x_{k + \frac 1 2} + \epsilon = x_r^- \quad
		x_r^+ = x_{k+\frac 1 2} + 3 \epsilon
	\]
	to form new cells around \[
	x_l = x_{k+ \frac 1 2} - 2 \epsilon \quad 
	x_m = x_{k + \frac 1 2}\quad
	 x_r = x_{k + \frac 1 2} + 2 \epsilon.\] 
	These cells are initialized with the mean values of $\tilde u^n(x)$ in these cells
		\[
				v_l = \frac{1}{2 \epsilon} \int_{x_l^-}^{x_l^+} \tilde u^n(x) \vd x \quad v_m = \frac{1}{2 \epsilon} \int_{x_m^-}^{x_m^+} \tilde u^n(x) \vd x \quad v_r = \frac{1}{2 \epsilon} \int_{x_r^-}^{x_r^+} \tilde u^n(x) \vd x.
		\]
		After one step of calculations we can take the entropy dissipation of the Lax-Friedrichs Scheme in the middle cell
		\[
			s_{k+\frac 1 2} = U\left(\frac {v_l + v_r}{2} + \lambda \frac {f(v_l) - f(v_r)}{2} \right) - \frac{U(v_l) + U(v_r))} 2 + \lambda \frac{F(v_l) - F(v_r)}{2}
		\]
		as an approximation of the true entropy dissipation at this edge.
		The value of $\epsilon$ is not critical in this calculation, and we can pass to the limit $\epsilon \to 0$ to find
		\begin{align*}
			s_{k + \frac 1 2} =  U\left(\frac {\tilde u^n_{k}\left(x_{k+ \frac 1 2}\right) + \tilde u^n_{k + 1}\left(x_{k+\frac 1 2} \right)}{2} + \lambda \frac {f\left(\tilde u^n_{k}\left(x_{k+ \frac 1 2}\right)\right) - f\left(\tilde u^n_{k + 1}\left(x_{k+\frac 1 2} \right)\right)}{2} \right) 
			\\
			- \frac{U\left(\tilde u^n_{k}\left(x_{k+ \frac 1 2}\right)\right) + U\left(\tilde u^n_{k + 1}\left(x_{k+\frac 1 2} \right)\right)} 2 + \lambda \frac{F\left(\tilde u^n_{k}\left(x_{k+ \frac 1 2}\right)\right) - F\left(\tilde u^n_{k + 1}\left(x_{k+\frac 1 2} \right)\right)}{2}
		\end{align*} and proceed as before with the usual stepping and sup-mollification operators. One should note that as $\lambda$ is constant the time step used for this calculation tends to zero and hence this entropy inequality is of no use for the ENO-LxF scheme and only gives an estimate for the entropy dissipation. 
	\begin{figure}
		\resizebox{\textwidth}{!}{
	\begin{tikzpicture}{scale = 0.5}
		\draw [dashed] (-10, -4) -- (-10, 4); 
		\draw [dashed] (0, -4) -- (0, 4);
		\draw [dashed] (10, -4) -- (10, 4);
		\draw [dotted] (-0.33, -4) -- (-0.33, 4); 
		\draw [dotted] (0.33, -4) -- (0.33, 4);
		\draw [dotted] (-1, -4) -- (-1, 4);
		\draw [dotted] (1, -4) -- (1, 4);
		\draw (-10, -2) -- (0, -2);
		\draw (0, 2) -- (10, 2);
		\draw (-10, -3) -- (0, -1);
		\draw (0, 1) -- (10, 3);
		\draw (-0.33, 0) -- (0.33, 0);
		\draw (-1, -1.13) -- (-0.33, -1.13);
		\draw (0.33, 1.13) -- (1, 1.13);
		\draw (0, -4) node[below] {\Large $x_{k + \frac 1 2}$};
		\draw (-5, -4) node[below] {\Large $x_k$};
		\draw (5, -4) node[below] {\Large $x_{k+1}$};
		\draw (-5, 3) node[above] {\Large $\tilde u_k^n(x)$};
		\draw (5, 3) node[above] {\Large $\tilde u_{k+1}^n(x)$};
		\draw (-0.66, 3) node {\Large $v_l$};
		\draw (0, 4) node[above] {\Large $v_m$};
		\draw (0.66, 3) node{\Large $v_r$};
		\draw (-0.66, -4) node[below]{\Large $x_l$};
		\draw (0.66, -4) node[below]{\Large $x_r$};
		
	\end{tikzpicture}
	}
		\caption{Subdivision and averaging after an ENO reconstruction}
		\label{fig:LFcelldiv}
	\end{figure}
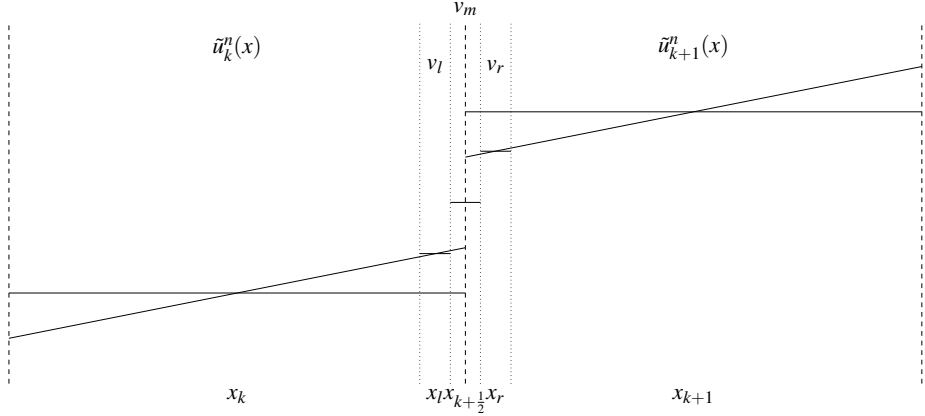
	\end{example}
\begin{remark}
	The aforementioned method is easily generalized to several space dimensions using a grid with tensor product structure and application of the presented methods in every direction, as also pointed out for correction procedure via reconstruction (CPR) Methods in their summation-by-parts (SBP) interpretation in \cite{ranocha2016summation}.
	\end{remark}
\begin{remark}
	The entropy inequality predictors are not discontinuity sensors. The function $\alpha$ should sense positions where entropy dissipation takes place. The entropy equality dictates that there has to be in fact a discontinuity if entropy is dissipated. The opposite implication does not hold. A discontinuity can be present in the solution but still no entropy is dissipated. An example of such behavior is the contact discontinuity present in some solutions to Riemann problems for the Euler equations. Therefore, a different approach, not equivalent to sensing discontinuities, is the aim of the scheme.
\end{remark}

\section{Numerical Tests}
\label{sec:3}
 \subsection{Numerical Tests for the Burgers equation}
Numerical tests were carried out for the new numerical flux composed of the entropy inequality predictor coupled to the convex combination flux. Sadly our scheme is not free of open parameters. The parameters $a, b$ were chosen as $a = 1/20, b = 1/100$ after some experiments. Wrong selection of $a$ results in a late or early detection of needed entropy dissipation. Those problems vanish for finer grids, as the entropy inequality predictor gives a refined distinction between conservation and entropy dissipation in this case. Still the optimal values for $a$ are only distributed over one order of magnitude. To high values of $b$ result in difficulties during time integration as this yields big Lipschitz constants for the resulting flux, while to small values result in a slow switching of the scheme between entropy conservation and extremal entropy dissipation. The values for $b$ are not as critical as values for $a$ and equally acceptable values span several orders of magnitude.
A second decision which had to be made concerned the entropy pair. The pair $U(u) = u^2/2, F(u) = u^3/3$ was chosen for this purpose.  We compare the new scheme directly to the Godunov scheme as the new scheme uses the Godunov scheme for dissipative regions. The Burgers equation was solved for $N = 50$ cells and periodic boundary conditions. A known good solution was calculated by a Godunov scheme with $N_{control} = 5000$ cells. Time integration was carried out using the SSPRK104 algorithm \cite{SSPRK}.

 \begin{figure}
 	\begin{subfigure}{0.49\textwidth}
 		\includegraphics[width=\textwidth]{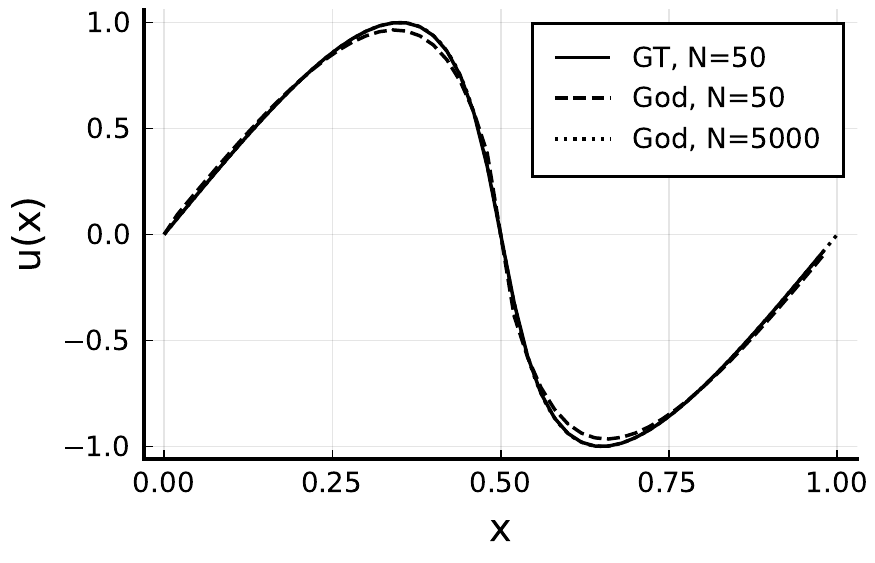}
 			\subcaption{Solutions at $t=0.22$}
 	\end{subfigure}
 		\hfill
 \begin{subfigure}{0.49\textwidth}
 	\includegraphics[width=\textwidth]{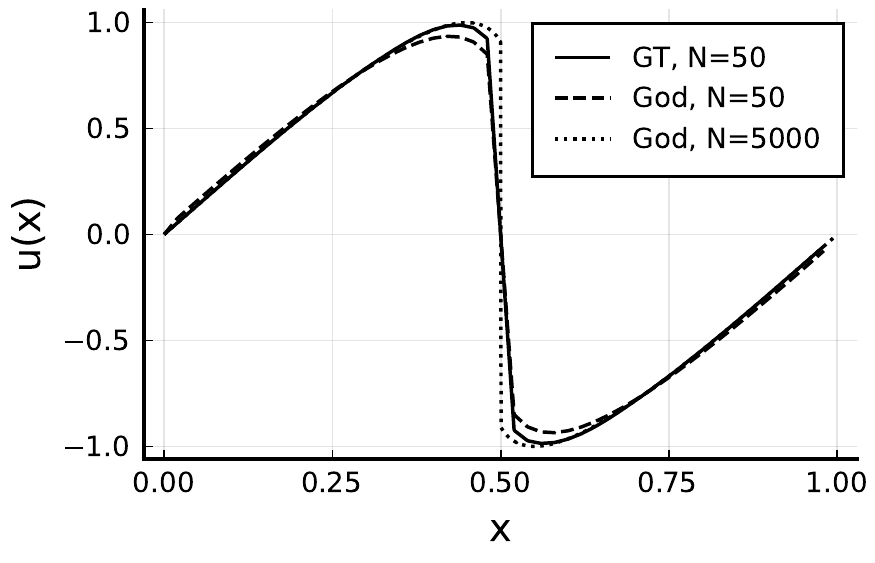}
 	\subcaption{Solutions at $t = 0.42$}	
 
 \end{subfigure}

\begin{subfigure}{0.49\textwidth}
	\includegraphics[width=\textwidth]{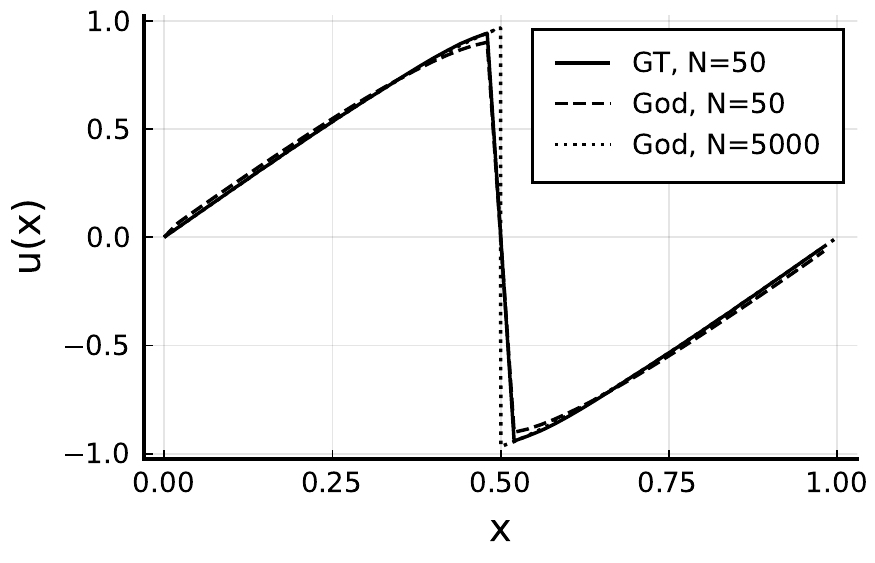}
	\subcaption{Solutions at $t = 0.62$}	
\end{subfigure}
		\hfill
\begin{subfigure}{0.49\textwidth}
	\includegraphics[width=\textwidth]{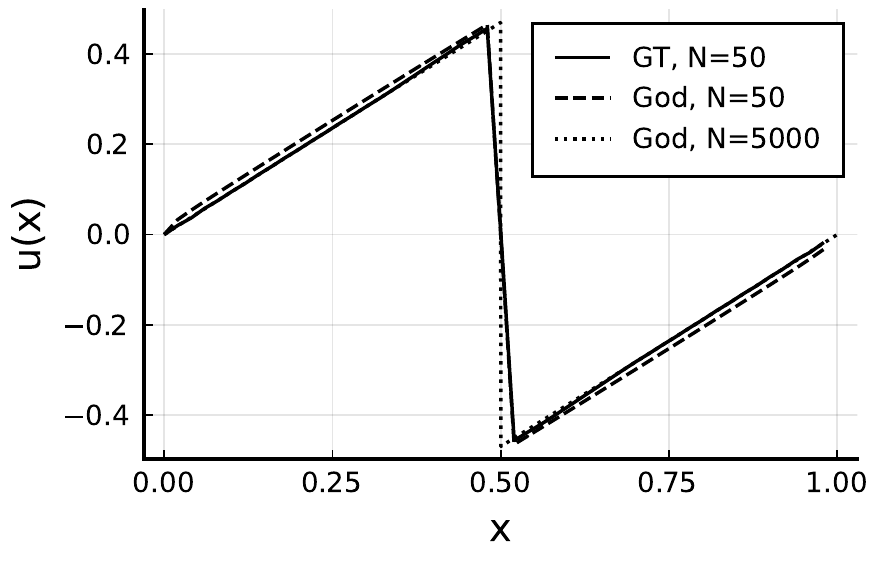}
	\subcaption{Solutions at $t = 1.82$}
	
\end{subfigure}
\begin{subfigure}{0.99\textwidth}
	\includegraphics[width=\textwidth]{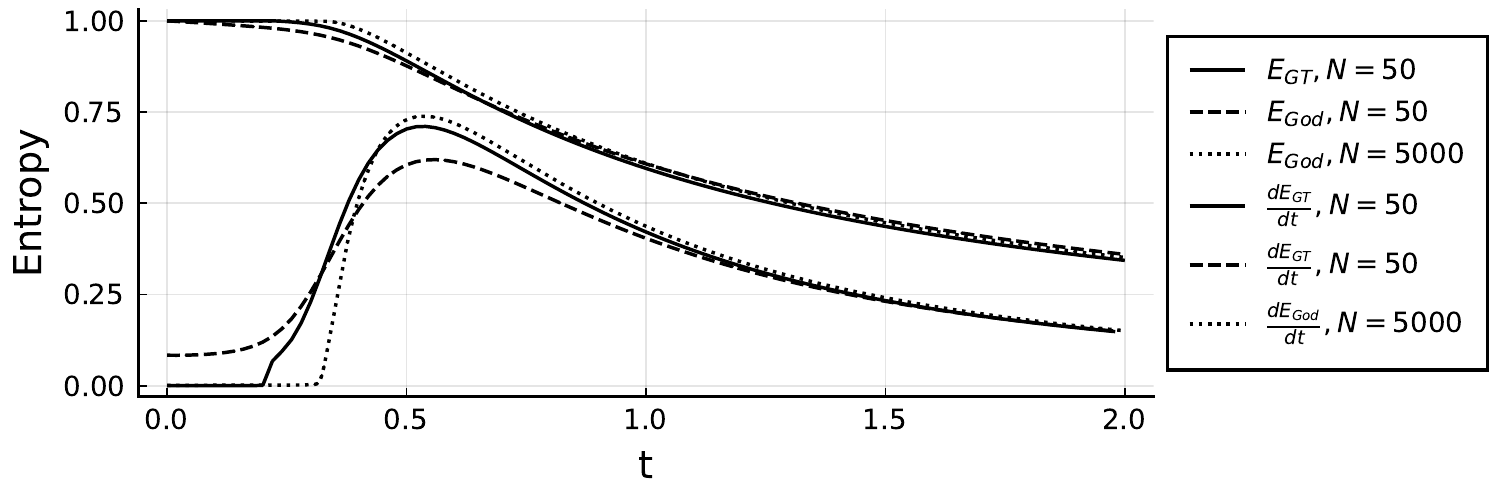}
	\subcaption{Total entropies and their derivatives with respect to time for all 3 solvers.}
	\label{fig:gtentro}
\end{subfigure}

\caption{Numerical experiment with the GT scheme. The entropy conservative flux is the eight order flux \cite{Tadmor87, LMR2002} while the classic Godunov scheme with exact Riemann solver was used as entropy dissipative flux. Time integration was carried out using the SSPRK104 method and a CFL number of $\lambda = 0.5$. The parameters of the Godunov entropy inequality predictor were $a = 1/20, b=1/100$. The cutted hat function used in the sup mollification was rescaled to fit the support of the hat into a $2p+1$ wide stencil with $p=8$, i.e. to fit the stencil of the high order flux.}
\label{fig:gtsol}
 	\end{figure}
Our numerical tests were carried out to test two assumptions.
\begin{itemize}
	\item The total entropy of the numerical solution of the GT scheme is a (good) approximation of the total entropy of the true solution. 
	\item The norm $\norm{u(\cdot, t) - u_{numeric}(\cdot, t)}$ is improved by our scheme over the error one gets from the Godunov scheme.
\end{itemize}
The first assumption seems to be true. By looking at figure \ref{fig:gtentro} the entropy of the GT scheme is, by construction, constant as long as u is smooth. The behavior is also desirable for non-smooth solutions as the numerical derivative of the total entropy approximates the exact derivative quite well. Interestingly the total entropy of the less dissipative GT scheme is smaller than the total entropy of the Godunov method for large times, which is the same for the exact solution.
Assumptions on the quality of the solution can be  made from the solution plots in \ref{fig:gtsol}. The smooth solutions show good correspondence between exact solution and the GT solution. In the discontinuous case the solution at the discontinuity corresponds to the solution of the Godunov method but is still significantly more exact in smooth areas.
\begin{figure}
	\begin{subfigure}{0.49 \textwidth}
		\includegraphics[width=\textwidth]{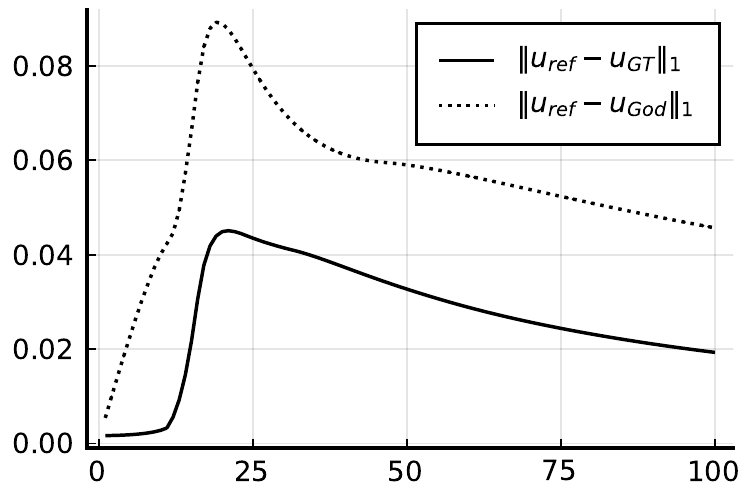}
		\subcaption{$\Leb_1$ norm}
		\end{subfigure}
	\hfill
	\begin{subfigure}{0.49 \textwidth}
		\includegraphics[width=\textwidth]{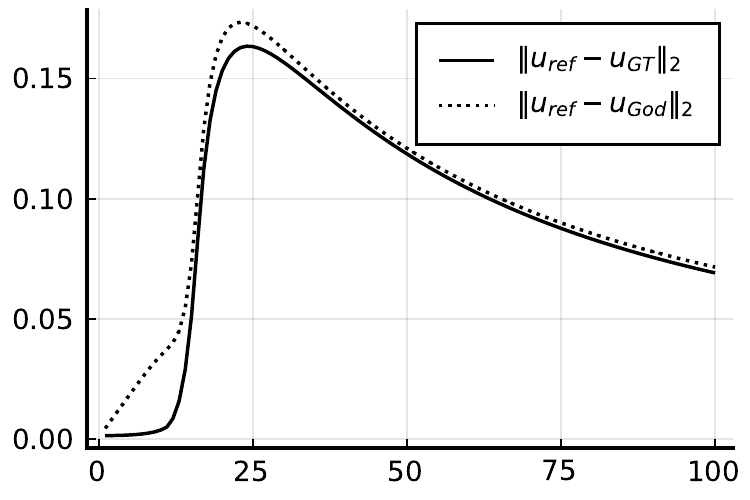}
		\subcaption{$\Leb_2$ norm}
		\end{subfigure}
	\caption{Error norms over time for the solutions in figure \ref{fig:gtsol} over time.}
	\label{fig:gterror}
	\end{figure}
After these qualitative assumptions some quantitative measurements were carried out in form of norms of the errors. While the $\Leb_1$ norm of the error was reduced for smooth and non-smooth solutions the $\Leb_2$ norm error for non-smooth solutions was only improved by a small amount as the shock is not better resolved than by the Godunov scheme. 
\begin{figure}
	\begin{subfigure}{0.49 \textwidth}
		\includegraphics[width=\textwidth]{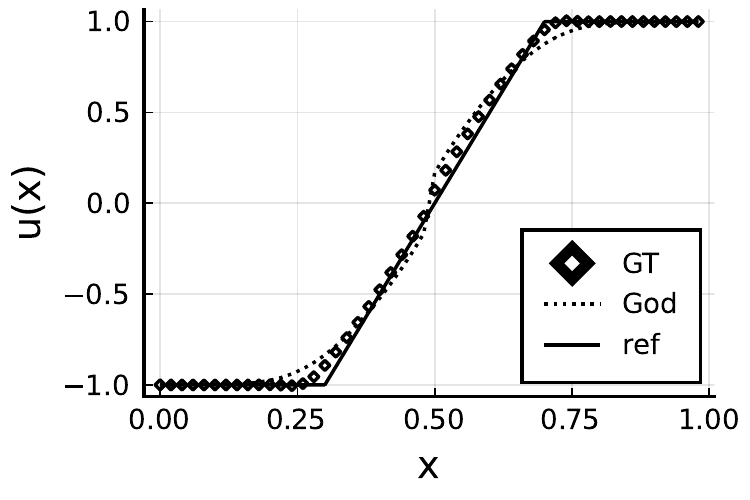}
		\subcaption{$t = 0.4$}
	\end{subfigure}
	\hfil
	\begin{subfigure}{0.49 \textwidth}
		\includegraphics[width=\textwidth]{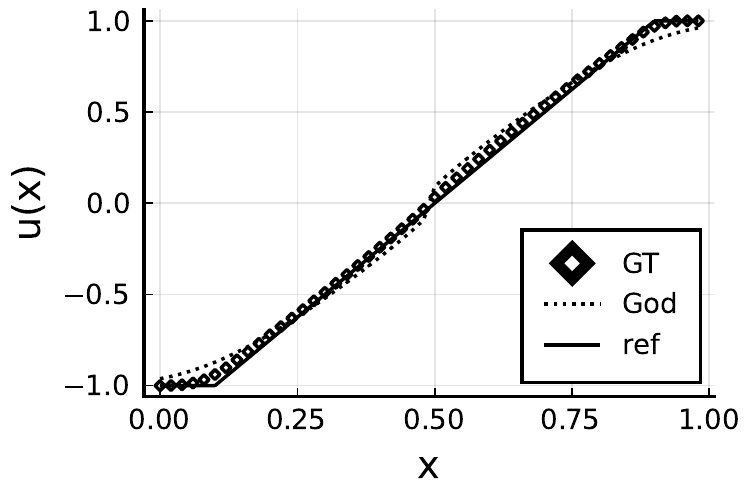}
		\subcaption{$t = 0.8$}
	\end{subfigure}
	\caption{Results for a Riemann Problem given by the initial condition $u_l = -1.0$, $u_r = 1.0$. A grid consisting of 50 cells was used in conjunction with a CFL number of $\lambda = 0.5$ and the SSPRK104 time integration method \cite{SSPRK}. Parameters where the same as in figure \ref{fig:gtsol}. The solution of the basic Godunov scheme was plotted as a reference for the possible sonic point glitch. The used GT scheme, based on the eight order entropy conservative flux, uses the same Godunov method as a low order flux. The sonic glitch is significantly reduced by the application of the GT scheme.}
	\label{fig:BurgSP}
\end{figure}
Several upwind schemes show glitches concerning rarefaction waves \cite{Tang2005Sonic}. The scheme was tested for this deficiency using a Riemann problem with $u_l = -1.0$ and $u_r = 1.0$ as initial condition and the results are shown in figure \ref{fig:BurgSP}. One could imagine that the sonic glitch, clearly present in the solution calculated by the Godunov method, will be also part of the solution calculated by the GT scheme. This is only partly the case. The strength of the sonic glitch is significantly reduced compared to the Godunov scheme. 
\subsection{Numerical Tests for the Euler equations of Gas-Dynamics}
After these promising results for the Burgers equation numerical tests were carried out for the Euler equations of gas dynamics
\[
	u =(\rho, \rho v, E) 
	\quad f(\rho, \rho v, E) = \begin{bmatrix} \rho v \\ \rho v^2 + p\\ v(E + p) \end{bmatrix} 
	\quad P = (\gamma - 1)\left (E - \frac 1 2 \rho v^2 \right)
\]
in conjunction with the LxFRI scheme and the ENO2LxF entropy inequality predictor. The physical entropy \cite{Harten83b, Tad2003}
\[
	U(\rho, \rho v, E) = - \rho S  \quad F(\rho, \rho v, E) = - \rho v S \quad S = \ln(p \rho^{- \gamma})
\]
was used in the entropy inequality predictor whereas the entropy conservative flux
\begin{align*}
	f^R(u_l, u_r) =  \begin{pmatrix}\hat \rho \hat u \\ \hat \rho \hat u^2 + \hat p_1 \\ \hat \rho \hat u \hat H  \end{pmatrix} \quad z = \sqrt {\frac \rho p} \begin{pmatrix} 1 \\ u \\ p \end{pmatrix} \\
	\hat \rho = \overline z_1 z_3^\text{ln} \quad \hat p_1 = \frac{\overline z_3}{\overline z_1} \quad \hat p_2 = \frac{\gamma + 1}{2 \gamma} \frac{z_3^\text{ln}}{z_1^\text{ln}} + \frac{\gamma-1}{2 \gamma}\frac{\overline z_3}{\overline z_1} \\
	\hat a = \sqrt{\frac{\gamma \hat p_2}{\hat \rho}} \quad
	\hat H = \frac{\hat a^2}{\gamma - 1} + \frac{\hat u^2}{2}
\end{align*}
developed by Ismail and Roe in \cite{IsmailRoe2009} that conserves the selected entropy
was used for the entropy conservative part of the scheme. This flux was selected as it is also used in several other publications \cite{FMT2012, Fisher2013Fluxdiff}. Other options, including fluxes that also conserves the kinetic energy, are possible \cite{ranocha2018comparison}.
The parameters $a = 1/1000$ and $b = 1/1000$, that were determined experimentally as before, were used. A new set of parameters is needed as a different entropy inequality predictor is used whose typical amplitudes and offset are different. Different entropies can also influence these parameters and an analysis giving explicit formulas is planed for a future publication. Reference solutions were calculated by the LMRLxFRI scheme with $N=1600$ points of order 6 and using SSPRK104 for time integration. First the ability of the scheme to resolve shocks was tested by the Shu-Osher test case number 6 from \cite{SO1989} given by the initial conditions
\begin{align*}
	\rho_0(x, 0) = \begin{cases}3.857153  \\ 1 + \epsilon \sin(5 x)  \end{cases} 
	\quad 
	v_0(x, 0) = \begin{cases} 2.629  \\ 0  \end{cases}
	p_0(x, 0) = \begin{cases} 10.333 & x < 1 \\ 1 & x \geq 1 \end{cases}
\end{align*}
The results can examined in figure \ref{fig:shuosh6}. A second experiment was carried out to demonstrate the ability of the scheme to achieve high order in smooth areas. The initial condition
\[
	\rho_0(x, 0) = 3.857153 + \epsilon \sin(2 x)  \quad v_0(x, 0) =  2.0 \quad p_0(x, 0) = 10.33333.
\]
is a density variation that is carried downstream to the right and the results from the convergence analysis are shown in figure \ref{fig:convana}.
\begin{figure}
	\begin{subfigure}{0.49\textwidth}
			\includegraphics[width=\textwidth]{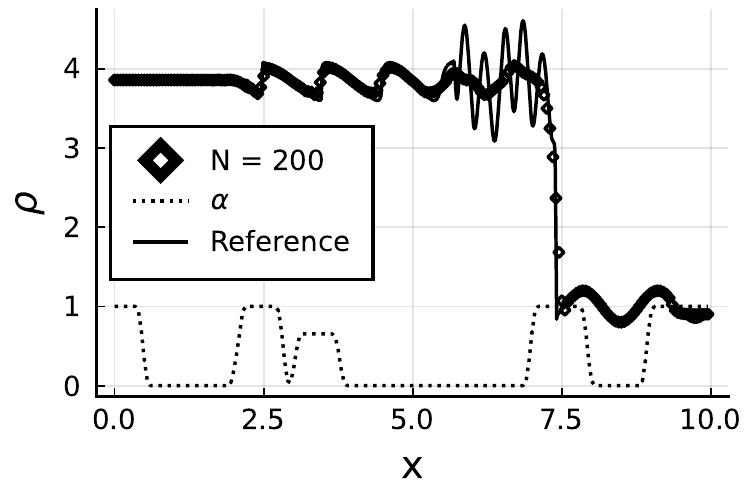}
		\subcaption{mass density}
	\end{subfigure}
	\hfill
	\begin{subfigure}{0.49\textwidth}
	\includegraphics[width=\textwidth]{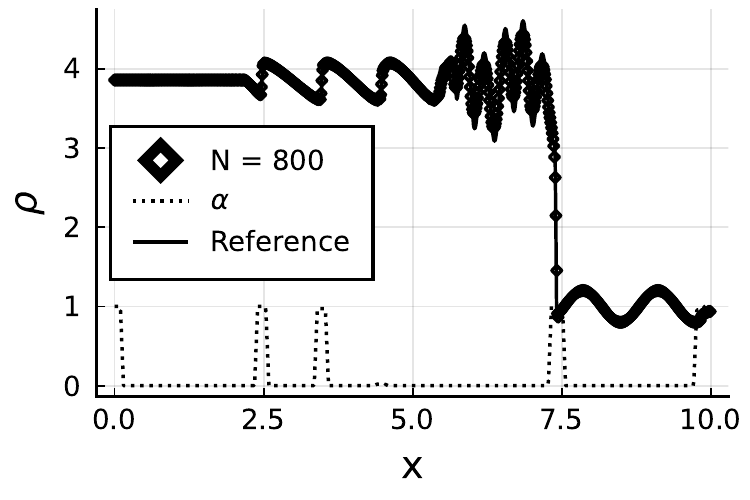}
	\subcaption{mass density}
\end{subfigure}
	\begin{subfigure}{0.49\textwidth}
			\includegraphics[width= \textwidth]{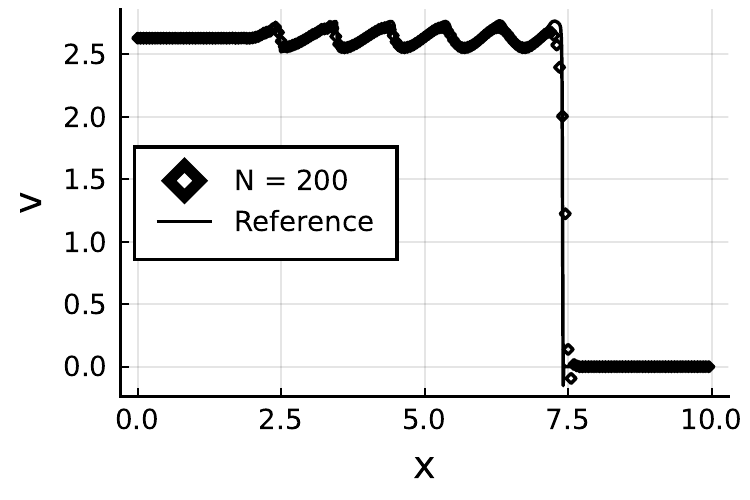}
			\subcaption{velocity}
	\end{subfigure}
\hfill
	\begin{subfigure}{0.49\textwidth}
	\includegraphics[width= \textwidth]{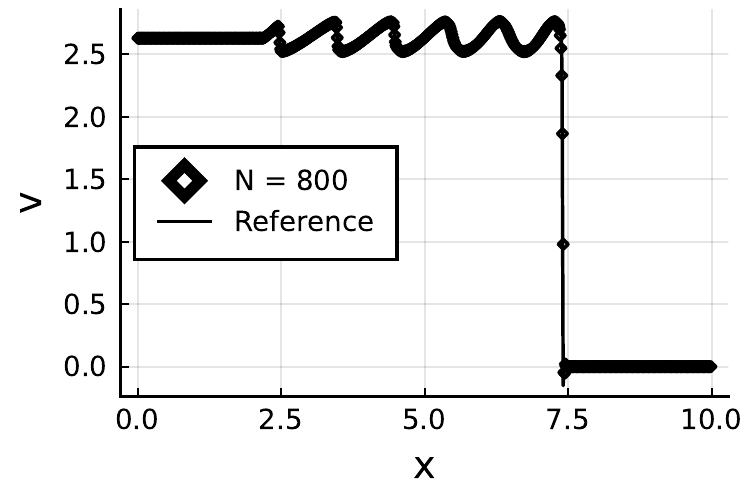}
	\subcaption{velocity}
\end{subfigure}
	\begin{subfigure}{0.49\textwidth}
			\includegraphics[width=\textwidth]{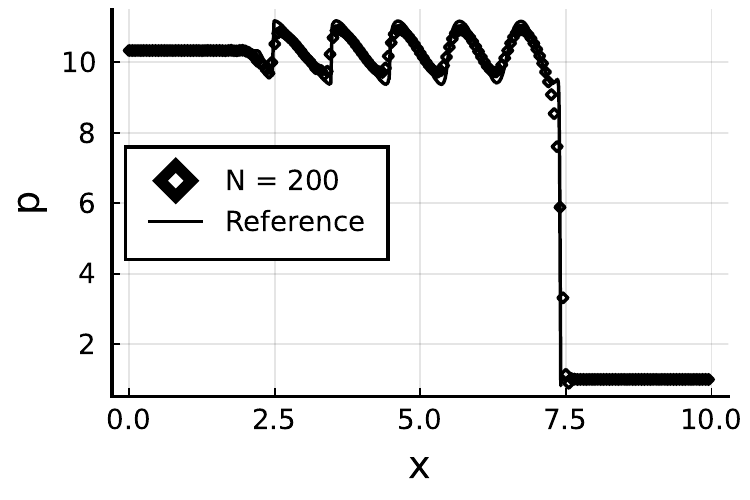}
			\subcaption{pressure}
	\end{subfigure}
	\hfill
	\begin{subfigure}{0.49\textwidth}
		\includegraphics[width=\textwidth]{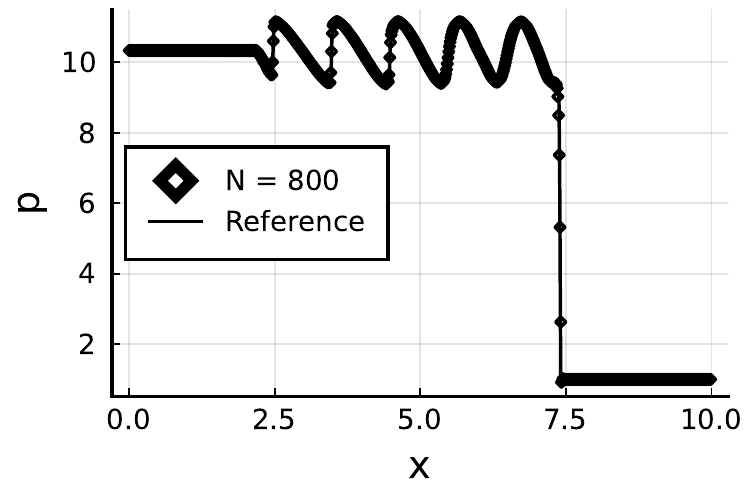}
		\subcaption{pressure}
	\end{subfigure}
	\caption{Shu-Osher testcase at $t = 1.8$. LMRLxFRI scheme of Order 6. The entropy conservative flux is the entropy conservative flux from \cite{IsmailRoe2009} while the Lax-Friedrich flux was used as entropy dissipative flux. Time integration was carried out using the SSPRK104 method and a CFL number of $\lambda = 0.1$. The parameters of the ENO2-Lax-Friedrichs entropy inequality predictor were $a = 1/1000, b=1/1000$. The cutted hat function used in the sup mollification was rescaled to fit the support of the hat into a $2p+1$ wide stencil with $p=6$, i.e. to fit the stencil of the high order flux. The values of $\alpha_{k+\frac 1 2}$ were also plotted.}
	\label{fig:shuosh6}
\end{figure}
\begin{figure}
	\centering
	\begin{subfigure}{0.49\textwidth}
		\includegraphics[width=\textwidth]{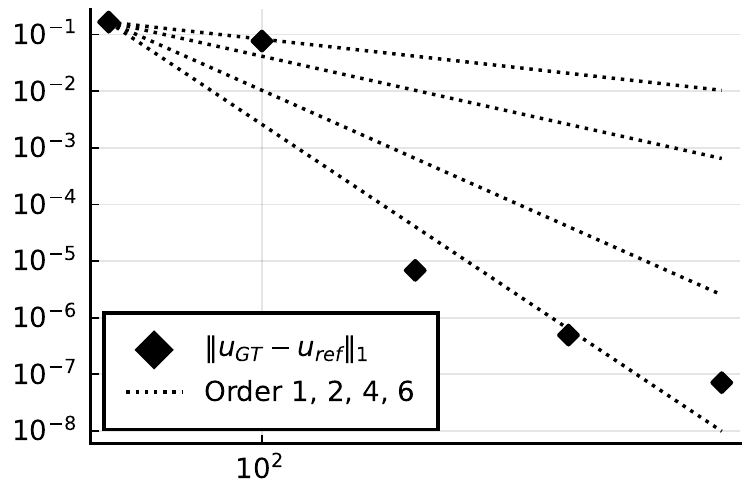}
		\subcaption{$\Leb_1$ Error for $N \in \sset{50, 100, 200, 400, 800}$ Points. Please note that the fourth order strong stability preserving Runge-Kutta method used for time-integration limits the achievable order to 4.}
	\end{subfigure}
	\hfill
	\begin{subfigure}{0.49\textwidth}
		\includegraphics[width=\textwidth]{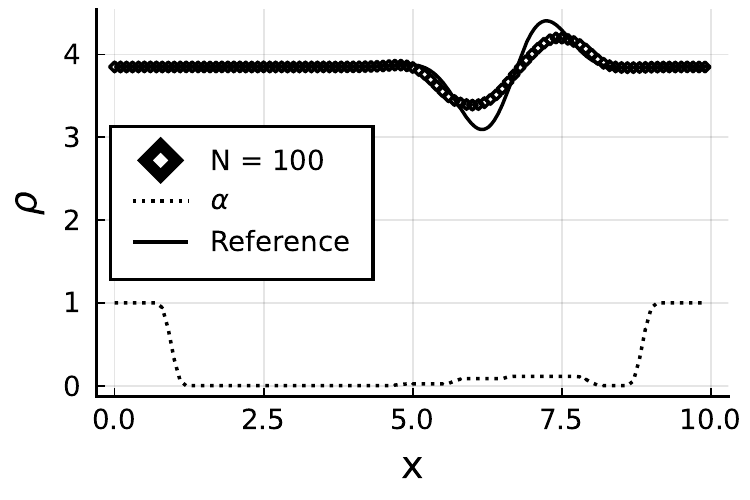}
		\subcaption{Solution for $N = 100$ points. The Entropy inequality predictor (miss) detects an entropy dissipating shock and partially activates the flux to dissipate entropy.}
	\end{subfigure}
	\begin{subfigure}{0.49\textwidth}
		\includegraphics[width=\textwidth]{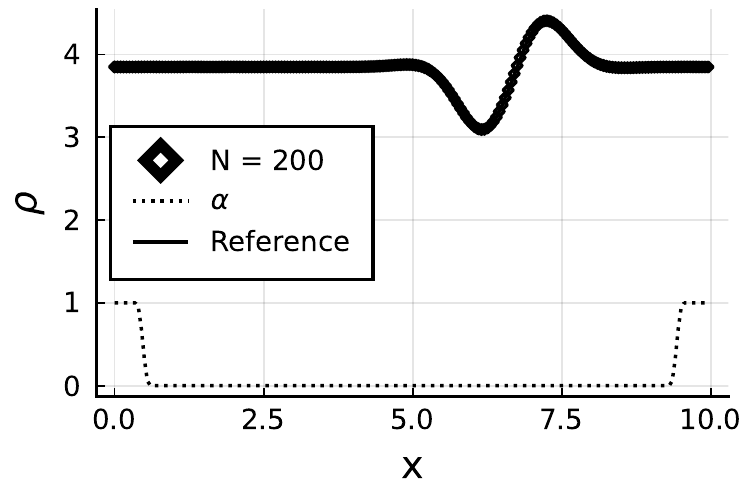}
		\subcaption{Solution for $N = 200$ points. The higher resolution deactivates the entropy inequality predictor and the entropy conservative high order flux reduces the entropy dissipation to zero.}
	\end{subfigure}
	\caption{convergence analysis for the Euler equations. LMRLxFRI Scheme of order 6. The same parameters and fluxes as in figure \ref{fig:shuosh6} were used.}
	\label{fig:convana}
\end{figure}

\section{Conclusion}
\label{sec:concl}

We first looked at some numerical solutions to hyperbolic conservation laws and saw that the entropy inequality is not enough to guarantee high quality solutions. Afterwards a new philosophy for the construction of schemes was proposed as they should satisfy the Dafermos entropy criterion and the entropy equality for smooth solutions. We then constructed such a solver by the hybrid usage of entropy conservative and entropy dissipative fluxes. Numerical experiments showed that having no entropy dissipation for smooth solutions, as motivated by the entropy equality, and enough entropy dissipation in discontinuous areas by the Godunov or respective LxF scheme provides a scheme with improved accuracy in smooth areas over the Godunov scheme and an accuracy not worse than the Godunov or respective LxF scheme in non-smooth areas. This can be seen as an improvement over prior attempts of using the Dafermos criterion for numerical schemes as in \cite[Chapter 9.2]{ranocha2018thesis} where excessive dissipation in smooth areas lead to bad solutions.
The primary difference being that the stencil selector proposed in \cite[Chapter 9.2]{ranocha2018thesis} also tried to dissipate the maximum amount of entropy in smooth areas while in fact the analytic theory in form of the entropy equality dictates the conservation of entropy as the maximum allowable reduction of entropy in this case. Research is ongoing concerning the improvement of stencil selection algorithms in reconstruction based methods by taking into account not only the maximum entropy dissipation but also the entropy equality for smooth areas. The methods that were constructed to calculate the coefficient $\alpha$ could be used also in methods based on steered dissipation as for example in \cite[chapter 11]{glaubitz2020thesis}. Better $\alpha$ distributions could on the other hand greatly enhance the abilities of the constructed schemes. An algorithm based on artificial intelligence has been tested by the author and a preprint concerning several other algorithms to calculate $\alpha$ is available. Practical applications of finite volume methods are mostly often multidimensional problems therefore a future publication concerning this scheme will generalize the presented method to multiple space dimensions on unstructured grids.

\section*{Acknowledgements}
Simon Klein would like to thank Thomas Sonar and Marko Stautz for their support during the preparation of this manuscript. Mr.~Klein would further like to thank the anonymous reviewer for his comments on the first draft of the manuscript, Hendrik Ranocha and Philipp Öffner for interesting discussions on entropy dissipative schemes, leading to an improved presentation of the material. The author was partially supported by the German Science Foundation (DFG) under the Grant SO 363/14-1.

\section{Bibliography}
\bibliographystyle{plainnat}
\bibliography{lit}
\end{document}